\documentclass[12pt]{article}

\usepackage[utf8]{inputenc}
\usepackage[TS1,T1]{fontenc}
\usepackage{xspace}

\usepackage{amssymb, amsmath}
\usepackage[all]{xy}

\usepackage[a4paper,hmargin=2cm,vmargin=2.5cm]{geometry}

\usepackage[english]{babel}
\usepackage{hyperref}
\usepackage{lmodern}

\usepackage{amsthm}
\newtheorem{theorem}{Theorem}[section]
\newtheorem{proposition}[theorem]{Proposition}
\newtheorem{lemma}[theorem]{Lemma}
\newtheorem{corollary}[theorem]{Corollary}

\theoremstyle{definition}

\newtheorem{definition}[theorem]{Definition}

\theoremstyle{remark}

\newcommand\bbbone{{ \mathchoice {1\mskip-4mu\mathrm{l} } {1\mskip-4mu\mathrm{l} }{1\mskip-4.5mu\mathrm{l} } {1\mskip-5mu\mathrm{l}} }}

\makeatletter
\newcommand*{\remove}{%
  \mathpalette\@remove
}
\def\@remove#1#2{
\def\@removesymbol{\boldsymbol\backslash}%
  \mathord{%
    \rlap{%
      \settowidth\dimen@{$\m@th#1{#2}$}%
      \kern.5\dimen@
      \settowidth\dimen@{$\m@th#1{\@removesymbol}$}%
      \kern-.5\dimen@
      $\m@th#1{\@removesymbol}$%
    }%
    {#2}%
  }%
}
\makeatother

\newcommand\dd{\text{\textup{d}}} 
\newcommand\ds{\text{\textup{s}}} 
\newcommand\hd{\widehat{\dd}} 

\newcommand{\grast}{\bullet} 
\newcommand\cdotaction{\mathord{\cdot}} 
\newcommand\exter{{\textstyle\bigwedge}} 
\newcommand\ordwedge{\mathord{\wedge}} 
\newcommand\ad{{\text{\textup{ad}}}} 
\newcommand\sign{{\text{\textup{sign}}}}
\newcommand\hor{{\text{\textup{hor}}}} 
\newcommand\inv{{\text{\textup{inv}}}} 
\newcommand\der{{\text{\textup{Der}}}} 
\newcommand\ensvide{{\varnothing}} 
\newcommand\Id{{\text{\textup{Id}}}} 
\newcommand\loc{{\text{\textup{loc}}}}
\newcommand\Ad{{\text{\textup{Ad}}}} 
\newcommand\lie{{\text{\textup{Lie}}}}
\newcommand\equ{{\text{\textup{equ}}}} 

\newcommand{\hodgeast}{\mathop{\star}}

\newcommand{\tla}{{\lienotation{TLA}}}
\newcommand{\halpha}{\widehat{\alpha}}
\newcommand{\raR}{{\widetilde{R}}}
\newcommand{\hg}{{\widehat{g}}}
\newcommand{\maxinner}{{\text{m.i.}}}

\newcommand{\inner}{{\text{\textup{inner}}}}
\newcommand{\algebraic}{{\text{\textup{alg.}}}}
\newcommand{\dR}{{\text{dR}}}

\newcommand{\lfc}{\mathfrak{a}} 

\DeclareMathOperator{\Aut}{Aut} 
\DeclareMathOperator{\End}{End} 
\DeclareMathOperator{\tr}{tr} 

\newcommand\varnotation[1]{{\mathcal{#1}}}
\newcommand\algnotation[1]{{\mathbf{#1}}}
\newcommand\lienotation[1]{{\mathbf{\mathsf{#1}}}}
\newcommand\grnotation[1]{{\mathsf{#1}}}

\newcommand\varA{{\varnotation{A}}}
\newcommand\varE{{\varnotation{E}}}

\newcommand\varL{{\varnotation{L}}}
\newcommand\varM{{\varnotation{M}}}
\newcommand\varP{{\varnotation{P}}}

\newcommand\algA{{\algnotation{A}}}

\newcommand\algzero{{\grnotation{0}}}

\newcommand\lieA{{\lienotation{A}}}

\newcommand\lieL{{\lienotation{L}}}

\newcommand\kg{{\mathfrak g}}

\newcommand\kD{{\mathfrak D}}
\newcommand\kS{{\mathfrak S}} 
\newcommand\kU{{\mathfrak U}}
\newcommand\kX{{\mathfrak X}}
\newcommand\kY{{\mathfrak Y}}
\newcommand\ksl{{\mathfrak{sl}}}
\newcommand\ku{{\mathfrak{u}}}

\newcommand\gC{{\mathbb C}}

\newcommand\sfX{{\mathsf X}}

\newcommand\caF{{\mathcal F}}
\newcommand\caH{{\mathcal H}}
\newcommand\caL{{\mathcal L}}
\newcommand\caN{{\mathcal N}}
\newcommand\caZ{{\mathcal Z}}

\hypersetup{
pdfauthor={Cédric Fournel, Serge Lazzarini, Thierry Masson},
pdftitle={Local description of generalized forms on transitive Lie algebroids and applications},
pdfsubject={Lie algebroids, connections, cohomology},
pdfcreator={pdflatex},
pdfproducer={pdflatex},
pdfkeywords={Lie algebroids}
}

\hypersetup{
plainpages=false,
colorlinks=true,
linkcolor=black, 
anchorcolor=black, 
citecolor=black, 
urlcolor=black, 
menucolor=black, 
filecolor=black, 
bookmarksopen=true,
bookmarksnumbered=true}

\numberwithin{equation}{section}

\begin{document}

\begin{center}%
 {\LARGE 
 Local description of generalized forms\\ on transitive Lie algebroids and applications 
 \par}%
 \vskip 3em%
 {\large
  \lineskip .75em%
      \begin{tabular}[t]{c}%
        Cédric Fournel \and Serge Lazzarini \and Thierry Masson
      \end{tabular}\par}%
      \vskip 1.5em%
Centre de Physique Théorique\footnote{Laboratoire affilié à la FRUMAM fédération de recherche 2291 du CNRS.},\\
Aix-Marseille Univ; CNRS, UMR 6207; Univ Sud Toulon Var;\\
13288 Marseille Cedex 9, France\par      
      \vskip 2.5em%
    {\large September 20, 2011 \par}
  \end{center}\par


\begin{abstract}
In this paper we study the local description of spaces of forms on transitive Lie algebroids. We use this local description to introduce global structures like metrics, $\hodgeast$-Hodge operation and integration along the algebraic part of the transitive Lie algebroid (its kernel). We construct a Čech-de~Rham bicomplex with a Leray-Serre spectral sequence. We apply the general theory to Atiyah Lie algebroids and to derivations on a vector bundle.

\end{abstract}

\vfill
\noindent CPT-P009-2011

\newpage

\tableofcontents

\bigskip

The study of Lie algebroids relies heavily on their local descriptions. These local descriptions are essentially  managed using geometric structures, like local fibered coordinates, (co)vector fields and tensors \cite{Fern02a,Mack05a}.

In this paper we are interested in the study of local description of differential forms on a transitive Lie algebroid $\lieA$ with kernel $\lieL$. We mainly consider two spaces of forms: forms with values in functions on the base manifold, and forms with values in $\lieL$. Our motivation is to establish a correspondence between differential complexes describing global objects on $\lieA$ and the corresponding differential complexes describing their local counterparts. In doing that, we have to manage the gluing relations at the intersection of local trivializations of the Lie algebroid in a convenient and manageable way. We shall not be concerned with more geometric aspects, like Lie groupoids and integration of Lie algebroids. We will avoid as much as possible the use of coordinates on the Lie algebroid. Only coordinates on the base manifold will be used to describe local forms and related structures.

One of the goals of this paper is to use these local descriptions of forms (Proposition~\ref{prop-relationtrivforms}) to introduce global constructions which are mandatory to develop and study Lagrangian field theories associated to generalized gauge fields in the context of (transitive) Lie algebroids. Concrete physical applications are out of the scope of the present paper. The aforementioned constructions are metrics and $\hodgeast$-Hodge operation (sub-sections~\ref{subsec-metrics} and \ref{subsec-hodgeoperators}), on the one hand, and orientability and (inner) integration (sub-section~\ref{subsec-innerorientationandintegration}), on the other hand. 

As noticed in \cite{Mass38}, the derivation-based noncommutative geometry studied in \cite{Mass14} and \cite{Mass15} is strongly related to the ``geometry'' involved in the space of forms of a transitive Lie algebroid with values in its kernel. This relation comes out from the fact that this noncommutative geometry is related to a well-identified Atiyah transitive Lie algebroid. 

In the present paper, we extend to more general transitive Lie algebroids the constructions described in \cite{Mass15} on metrics, integration along the algebraic components of this noncommutative geometry and a Leray-Serre spectral sequence associated to a Čech-de~Rham bicomplex.

The local description of forms on $\lieA$ with values in functions and the integration along the algebraic part have been studied in \cite{Kuba96a,MR1908998,MR2020382}. Here we generalize a part of these works by defining an integral along the ``fibre'' (algebraic or ``inner'' part of the transitive Lie algebroid) for forms with values in the kernel $\lieL$ (Definition~\ref{def-innerintegration}) and by studying the Čech-de~Rham bicomplex associated to this space of forms (Section~\ref{sec-cohomology}).

In Proposition~\ref{prop-relationtrivforms}, the fiber structure in Lie algebras of $\lieL$ in the local description of forms with values in $\lieL$ is taken into account using a purely algebraic map $\halpha_{j}^{\,i}$ defined in \eqref{eq-definitionalphahat}. The properties of this map, described in part in Proposition~\ref{prop-localdifferentialcommutes}, make possible many constructions presented in this paper.

In \cite{Kuba96a,MR1908998}, the structure which permits to define integration along the fiber is a non-singular cross-section $\varepsilon$ in $\exter^n \lieL$. The starting point of our definition is different. In Section~\ref{sec-metric-integration}, we use the notion of metric on $\lieL$. Then we associated to such a metric a global form of maximal inner degree (Proposition~\ref{prop-globalinnerform}). This form plays the role of a ``volume form'' for this fiber integration. This form is dual to $\varepsilon$ in a certain sense as explained. On the other hand, the notion of metric permits also to define Hodge operators.

In Section~\ref{sec-applicationAtiyah} we specify four constructions to Atiyah Lie algebroids for which the underlying geometry of the principal fiber bundle helps us to improve some results obtained in the general case. For instance, using integrations along the algebraic parts, we related in Proposition~\ref{prop-relationsdeRhamTLAAtiyah} the de~Rham calculus on the principal fiber bundle to the space of forms with values in the kernel. In Section~\ref{sec-Derivationsonavectorbundle} we improve some general results for the case of Lie algebroids of derivations on a vector bundle. In that situation we extend in \eqref{eq-definnerintegrationtrace} our notion of integration along the algebraic part, and we make apparent close relations with the noncommutative structures mentioned above.

\section{Differential forms and their local descriptions}

\subsection{Lie algebroids}

In this paper we use the notations introduced in \cite{Mass38}. We refer to \cite{Mack05a} for more developments on Lie algebroids.

We use the following algebraic version of the definition of Lie algebroids:

\begin{definition}
\label{def-liealgebroidalgebraic}
Let $\varM$ be a smooth manifold. A Lie algebroid $\lieA$ is a finite projective module over $C^\infty(\varM)$ equipped with a Lie bracket $[-,-]$ and a $C^\infty(\varM)$-linear Lie morphism, the anchor, $\rho : \lieA \rightarrow \Gamma(T\varM)$ such that
\begin{equation*}
[\kX, f \kY] = f [\kX, \kY] + (\rho(\kX)\cdotaction f) \kY
\end{equation*}
for any $\kX, \kY \in \lieA$ and $f \in C^\infty(\varM)$ where $\Gamma(T\varM)$ denotes as usual the space of smooth vector fields on $\varM$.
\end{definition}

In this paper, $C^\infty(\varM)$ stands for complex valued functions on $\varM$.

The requirement that $\lieA$ be a finitely generated projective module implies that there is a vector bundle $\varA \xrightarrow{\pi} \varM$ such that $\lieA = \Gamma(\varA)$. This makes this definition equivalent to the one using only geometric structures.

The notation $\lieA \xrightarrow{\rho} \Gamma(T\varM)$ will be used for a Lie algebroid $\lieA$ over the manifold $\varM$ with anchor $\rho$.

Our paper focuses on transitive Lie algebroids :
\begin{definition}
A Lie algebroid $\lieA \xrightarrow{\rho} \Gamma(T\varM)$ is transitive if $\rho$ is surjective.
\end{definition}

The following result is well-known \cite{Mack05a} and is a central component of our constructions:
\begin{proposition}[The kernel of a transitive Lie algebroid]
Let $\lieA$ be a transitive Lie algebroid. Then $\lieL = \ker \rho$ is a Lie algebroid with null anchor on $\varM$.
The vector bundle $\varL$ such that $\lieL = \Gamma(\varL)$ is a locally trivial bundle in Lie algebras. The Lie bracket on $\lieL$ is inherited from the one on the Lie algebra on which the fiber of this vector bundle is modelled.

One has the short exact sequence of Lie algebras and $C^\infty(\varM)$-modules
\begin{equation*}
\xymatrix@1{{\algzero} \ar[r] & {\lieL} \ar[r]^-{\iota} & {\lieA} \ar[r]^-{\rho} & {\Gamma(T\varM)} \ar[r] & {\algzero}}
\end{equation*}
$\lieL$ will be called the kernel of $\lieA$.
\end{proposition}

\subsection{Differential forms}

There are natural spaces of ``differential forms'' to be considered on a (transitive or not) Lie algebroid $\lieA$. They depend on the choice of a representation of $\lieA$ on a vector bundle. We are interested in two of them. In \cite{Mass38} we have introduced the following notations for them.

\begin{definition}
\label{def-formsvaluesfunctions}
Let $\lieA \xrightarrow{\rho} \Gamma(T\varM)$ be a Lie algebroid (not necessarily transitive).
We define $(\Omega^\grast(\lieA), \hd_\lieA)$ as the graded commutative differential algebra of forms on $\lieA$ with values in $C^\infty(\varM)$. 
\end{definition}

This space of forms is the natural one to be considered \cite{ArAbCrai09,Crai03a,CraiFern2006a,CraiFern2009a,Fern03a,Mack05a} in order to study the structural properties of $\lieA$. Considerations about its cohomology can be found in those references. 

\smallskip
As noticed in \cite{Mass38}, some notions of connections can be understood in terms of another differential calculus.

\begin{definition}
\label{def-formsvalueskernel}
Let $\lieA \xrightarrow{\rho} \Gamma(T\varM)$ be a transitive Lie algebroid, with $\lieL$ its kernel.
We define $(\Omega^\grast(\lieA, \lieL), \hd)$ as the graded differential Lie algebra of forms on $\lieA$ with values in the kernel $\lieL$, where $\lieA$ is represented on $\lieL$ by the usual adjoint representation.
\end{definition}

$\Omega^\grast(\lieA, \lieL)$ is a natural graded module on the graded commutative algebra $\Omega^\grast(\lieA)$. This product is compatible with the two differentials in the sense that for any $\eta_p \in \Omega^p(\lieA)$ and $\omega_q \in \Omega^q(\lieA, \lieL)$ one has $\hd (\eta_p \omega_q) = (\hd_\lieA \eta_p) \omega_q + (-1)^{p} \eta_p (\hd \omega_q)$.

\medskip
The aim of this paper is to describe forms belonging to the previously defined spaces by using local descriptions of transitive Lie algebroids in terms of trivial Lie algebroids. We will need a concrete description of the previously defined differential complexes in this particular case.

A trivial Lie algebroid is just the Atiyah transitive Lie algebroid associated to a trivial principal fiber bundle $\varM \times G$ where $\varM$ is a manifold and $G$ is a Lie group (see Section~\ref{sec-applicationAtiyah}). Denote by $\kg$ the Lie algebra of $G$. A convenient description of the associated Lie algebroid is given by the following construction. Consider the trivial vector bundle $\varA = T\varM \oplus (\varM \times \kg)$. The space of smooth sections $\lieA = \Gamma(\varA)$ of $\varA$ is a Lie algebroid for the anchor $\rho(X \oplus \gamma) = X$ and the bracket $[X \oplus \gamma, Y \oplus \eta] = [X,Y] \oplus (X \cdotaction \eta - Y \cdotaction \gamma + [\gamma,\eta])$ for any 
$X,Y \in \Gamma(T\varM)$ and $\gamma,\eta \in \Gamma(\varM \times \kg) = \lieL$, where $\lieL$ is the kernel.

We shall use the compact notation $\tla(\varM, \kg)$ to designate the trivial Lie algebroid obtained this way.

For $\lieA = \tla(\varM, \kg)$, the graded commutative differential algebra $(\Omega^\grast(\lieA), \hd_\lieA)$ identifies with the total complex of the bigraded commutative algebra $\Omega^\grast(\varM) \otimes \exter^\grast \kg^\ast$ equipped with the two differential operators
\begin{align*}
\dd : \Omega^\grast(\varM) \otimes \exter^\grast \kg^\ast &\rightarrow \Omega^{\grast+1}(\varM) \otimes \exter^\grast \kg^\ast\\
\ds : \Omega^\grast(\varM) \otimes \exter^\grast \kg^\ast &\rightarrow \Omega^\grast(\varM) \otimes \exter^{\grast+1} \kg^\ast
\end{align*}
where $\dd$ is the de~Rham differential on $\Omega^\grast(\varM)$, and $\ds$ is the Chevalley-Eilenberg differential on $\exter^\grast \kg^\ast$ so that $\hd_\lieA = \dd + \ds$.

The graded differential Lie algebra $(\Omega^\grast(\lieA, \lieL), \hd)$ identifies with the total complex of the bigraded Lie algebra $\Omega^\grast(\varM) \otimes \exter^\grast \kg^\ast \otimes \kg$ equipped with the differential $\dd$ and the Chevalley-Eilenberg  differential $\ds'$ on $\exter^\grast \kg^\ast \otimes \kg$ for the adjoint representation of $\kg$ on itself, so that $\hd = \dd + \ds'$. We shall use the compact notation $(\Omega^\grast_\tla(\varM,\kg), \hd_\tla)$ for this graded differential Lie algebra.

\subsection{Local trivialization of a transitive Lie algebroid}

In this subsection we fix notations about local descriptions of transitive Lie algebroids in terms of trivial Lie algebroids. As explained in details in \cite{Mack05a} and with similar notation, a transitive Lie algebroid $\lieA \xrightarrow{\rho} \Gamma(T\varM)$ with kernel $\lieL = \Gamma(\varL)$ can be described locally as a triple $(U, \Psi, \nabla^0)$ where
\begin{itemize}
\item $\Psi : U \times \kg \xrightarrow{\simeq} \varL_U$ is a local trivialization of $\varL$ as a fiber bundle in Lie algebras over the open subset $U \subset \varM$. In particular one has $[\Psi(\gamma), \Psi(\eta)] = \Psi([\gamma, \eta])$
 and this morphism gives rise to an isomorphism of Lie algebras and $C^\infty(U)$-modules $\Psi : \Gamma(U \times \kg) \xrightarrow{\simeq} \lieL_U$.

\item $\nabla^0 :  \Gamma(TU) \rightarrow \lieA_U$ is an injective morphism of Lie algebras and $C^\infty(U)$-modules compatible with the anchors: $[\nabla^0_X, \nabla^0_Y] = \nabla^0_{[X,Y]}$, $\nabla^0_{fX} = f \nabla^0_X$, and $\rho \circ \nabla^0_X = X$ for any $X,Y \in \Gamma(TU)$ and any $f \in C^\infty(U)$.

\item For any $X \in \Gamma(TU)$ and any $\gamma \in \Gamma(U \times \kg)$, one has $[\nabla^0_X, \iota \circ \Psi(\gamma) ] = \iota \circ \Psi(X \cdotaction \gamma)$.
\end{itemize}

Such a triple defines an isomorphism of Lie algebroids
\begin{align*}
S : \tla(U, \kg) &\xrightarrow{\simeq} \lieA_U
&
S(X \oplus \gamma) &= \nabla^0_X + \iota \circ \Psi(\gamma)
\end{align*}

\begin{definition}
A Lie algebroid atlas for $\lieA$ is a family of triples $\{(U_i, \Psi_i, \nabla^{0,i})\}_{i \in I}$ such that $\bigcup_{i \in I} U_i = \varM$ and each triple $(U_i, \Psi_i, \nabla^{0,i})$ is a local trivialization of $\lieA$.
\end{definition}

Then the family $(U_i, \Psi_i)$ realizes a system of local trivializations of the fiber bundle in Lie algebras $\varL$. On $U_{ij} = U_i \cap U_j \neq \ensvide$ one can define 
\begin{align*}
\alpha^{i}_{j} : U_{ij} &\rightarrow \Aut(\kg) & \alpha^{i}_{j} &= \Psi_i^{-1} \circ \Psi_j.
\end{align*}

Any element $\kX \in \lieA$ can be trivialized as $X^i \oplus \gamma^i \in \tla(U_i, \kg)$ over each $U_i$ with the defining relation $S_i(X^i \oplus \gamma^i) = \kX_{|U_i}$. The local vector fields $X^i$ are the restriction of the global vector field $X = \rho(\kX)$, so that on $U_{ij}$ one has ${X^i}_{|U_{ij}} = {X^j}_{|U_{ij}}$. This permits one to denote all of them by $X$. There exists $\ell_{ij} \in \Omega^1(\lieA_{U_{ij}}, \lieL_{U_{ij}})$ such that 
\begin{align*}
\nabla^{0,j}_X &= \nabla^{0,i}_X + \iota \circ \ell_{ij}(X) 
& &\text{and} &
\Psi_i(\gamma^i) &= \Psi_j(\gamma^j) + \ell_{ij}(X). 
\end{align*}

Anticipating on some Čech cohomology considerations, we make a distinction between $U_{ij}$ and $U_{ji}$ for $i \neq j$. Then one defines without any ambiguity 
\begin{equation*}
\chi_{ij} = \Psi_i^{-1} \circ \ell_{ij} \in \Omega^1(U_{ij}) \otimes \kg 
\end{equation*}
so that 
\begin{equation*}
\gamma^i = \alpha^{i}_{j}(\gamma^j) + \chi_{ij}(X) 
\end{equation*}

On $U_{ijk} = U_i \cap U_j \cap U_k \neq \ensvide$, one has the two cocycle relations:
\begin{align*}
\alpha^{i}_{k} &= \alpha^{i}_{j} \circ \alpha^{j}_{k}
&
\chi_{ik} &= \alpha^{i}_{j} \circ \chi_{jk} + \chi_{ij}
\end{align*}
The composite map $X \oplus \gamma^i \mapsto X \oplus \gamma^j \mapsto X \oplus \gamma^i$ on $U_{ij}$ gives
\begin{align*}
\alpha^{i}_{j} \circ \alpha^{j}_{i} &= \Id \in \Aut(\kg)
&
\alpha^{i}_{j} \circ \chi_{ji} + \chi_{ij} &= 0
\end{align*}
These expressions are compatible with the previous ones upon defining
\begin{align*}
\alpha^{i}_{i} &= \Id \in \Aut(\kg)
&
\chi_{ii} &= 0
\end{align*}

\subsection{Local trivializations of forms}

Using a local description of a transitive Lie algebroid, we can locally describe a form using the following definition.

\begin{definition}
\label{def-traivializationofforms}
Let $(U, \Psi, \nabla^0)$ be a local trivialization of $\lieA$. To any $q$-form $\omega \in \Omega^q(\lieA, \lieL)$ we define a local $q$-form $\omega_{\loc} \in \Omega^q_\tla(U,\kg)$ by
\begin{equation*}
\omega_{\loc} = \Psi^{-1} \circ \omega \circ S
\end{equation*}
\end{definition}

Given a Lie algebroid atlas for $\lieA$, one associates to $\omega \in \Omega^q(\lieA, \lieL)$ a family of local forms $\omega_{\loc}^i \in \Omega^q_\tla(U_i,\kg)$.

For any $\kX_k \in \lieA$ with $1\leq k \leq q$, let $X_k \oplus \gamma^i_k \in \tla(U_i, \kg)$ be its family of trivializations. Then on any $U_{ij} = U_i \cap U_j \neq \ensvide$ one has
\begin{equation*}
\omega_{\loc}^i(X_1 \oplus \gamma^i_1, \ldots, X_q \oplus \gamma^i_q) = \alpha^{i}_{j} \circ \omega_{\loc}^j(X_1 \oplus \gamma^j_1, \ldots, X_q \oplus \gamma^j_q).
\end{equation*}
Notice that $s_{i}^{j} = S_j^{-1} \circ S_i : \tla(U_{ij}, \kg) \xrightarrow{\simeq} \tla(U_{ij}, \kg)$ is an isomorphism of (trivial) Lie algebroids and that the previous relation takes the compact form
\begin{equation}
\label{eq-changelocaltrivializationforms}
\omega_{\loc}^i = \alpha^{i}_{j} \circ \omega_{\loc}^j \circ s_{i}^{j}
\end{equation}
Let us define $\halpha_{j}^{\,i}$ on $\Omega^q_\tla(U_{ij},\kg)$ by
\begin{equation}
\label{eq-definitionalphahat}
\halpha_{j}^{\,i}(\omega_{\loc}^j) = \alpha^{i}_{j} \circ \omega_{\loc}^j \circ s_{i}^{j}.
\end{equation}

\begin{proposition}
\label{prop-relationtrivforms}
A family of local forms $\{\omega_{\loc}^i\}_{i \in I}$ with $\omega_{\loc}^i \in \Omega^\grast_\tla(U_i,\kg)$ is a system of trivializations of a global form $\omega \in \Omega^\grast(\lieA, \lieL)$ if and only if 
\begin{equation}
\label{eq-relationtrivforms}
\halpha_{j}^{\,i}(\omega_{\loc}^j) = \omega_{\loc}^i
\end{equation}
for any $i,j$ such that $U_{ij} \neq \ensvide$.
\end{proposition}

In the particular case of a $1$-form $\omega \in \Omega^1(\lieA, \lieL)$, one can write locally $\omega_\loc(X \oplus \gamma) = a(X) + b(\gamma)$ for any $X \oplus \gamma \in \tla(U, \kg)$, where $a \in \Omega^1(U)\otimes \kg$ and $b \in (\exter^1 \kg^\ast) \otimes \kg$. On $U_{ij}$, with $\omega_\loc^i =  a^i + b^i$ and $\omega_\loc^j =  a^j + b^j$ the local expressions of $\omega$ on $U_i$ and $U_j$ respectively, one has the transition formulas
\begin{align}
\label{eq-gluingrelationslocalexpressiononeform}
a^i &= \alpha^{i}_{j} \circ a^j - b^i \circ \chi_{ij}
&
b^i &= \alpha^{i}_{j} \circ b^j \circ \alpha^{j}_{i}
\end{align}

Recall that a connection on a transitive Lie algebroid $\lieA \xrightarrow{\rho} \Gamma(T \varM)$ is a splitting $\nabla : \Gamma(T \varM) \rightarrow \lieA$ as $C^\infty(\varM)$-modules of the short exact sequence
\begin{equation*}
\xymatrix@1@C=25pt{{\algzero} \ar[r] & {\lieL} \ar[r]^-{\iota} & {\lieA} \ar[r]_-{\rho} & {\Gamma(T \varM)} \ar[r] \ar@/_0.7pc/[l]_-{\nabla} & {\algzero}}
\end{equation*}
Then one can associate to $\nabla$ a $1$-form $\lfc^\nabla \in \Omega^1(\lieA, \lieL)$ uniquely defined by
\begin{equation*}
\kX = \nabla_X - \iota \circ \lfc^\nabla(\kX)
\end{equation*}
This $1$-form is normalized by $\lfc^\nabla \circ \iota(\ell) = -\ell$ for any $\ell \in \lieL$. In the following we will call it the connection $1$-form of the connection $\nabla$.

In a trivialization $(U, \Psi, \nabla^{0})$ of $\lieA$, the $1$-form $\lfc^\nabla$ can be trivialized as $\lfc^\nabla_{\loc}(X \oplus \gamma) = A(X) - \gamma$ with $A \in \Omega^1(U)\otimes \kg$. Notice that as an element of $\tla(U, \kg)$, one has $S^{-1}(\nabla_X) = X \oplus A(X)$.

With obvious notations, on $U_i \cap U_j \neq \ensvide$, we can use \eqref{eq-gluingrelationslocalexpressiononeform} with $b_i(\gamma) = - \gamma$ to get the gluing rules
\begin{equation}
\label{eq-gluingrelationlocalexpressionconnection}
A_i = \alpha^{i}_{j} \circ A_j + \chi_{ij}.
\end{equation}

Conversely, a family of $1$-forms $A_i \in \Omega^1(U_i)\otimes \kg$ satisfying \eqref{eq-gluingrelationlocalexpressionconnection} defines a connection on $\lieA$. As for local description of connection $1$-form on principal fiber bundle, the gluing relations require an inhomogeneous term. In order to get homogeneous relations, we should consider the local forms $\lfc^\nabla_{\loc, i}$ instead of the $A_i$'s.

\smallskip
Let us give general and useful properties of $\halpha_{j}^{\,i}$.

\begin{proposition}
\label{prop-localdifferentialcommutes}
One has
\begin{equation*}
\hd_\tla \omega_{\loc} = \Psi^{-1} \circ (\hd\omega) \circ S
\end{equation*}
and, on $U_{ij} = U_i \cap U_j \neq \ensvide$,
\begin{equation*}
\hd_\tla \omega_{\loc}^i =
\halpha_{j}^{\,i} \left(\hd_\tla \omega_{\loc}^j \right).
\end{equation*}

The map $\halpha_{j}^{\,i} : \Omega^\grast_\tla(U_{ij},\kg) \rightarrow \Omega^\grast_\tla(U_{ij},\kg)$ is an isomorphism of graded differential Lie algebras.
\end{proposition}

\begin{proof}
These are straightforward computations using the definitions and the properties of $\hd$, $\hd_\tla$, $\Psi$, $\nabla^0$ and $S$. For the second relation, one needs the easy to establish relation
\begin{equation*}
X \cdotaction \alpha^{i}_{j}(\eta) + [ \gamma^i, \alpha^{i}_{j}(\eta)] = \alpha^{i}_{j}(X \cdotaction \eta + [\gamma^j, \eta]) 
\end{equation*}
for any $\eta \in \Gamma(U_{ij} \times \kg)$.
\end{proof}

For the case of $C^\infty(\varM)$-valued forms, we shall also denote by $\halpha_{j}^{\,i} : \Omega^\grast_\tla(U_{ij}) \rightarrow \Omega^\grast_\tla(U_{ij})$ the isomorphism of differential graded algebras $\halpha_{j}^{\,i}(\omega^j_{\loc}) = \omega^j_{\loc} \circ s_{i}^{j}$ performing the change of trivializations for local expressions of forms in $\Omega^\grast(\lieA)$ which are defined in the same way as in Definition~\ref{def-traivializationofforms} without the need for the composition with $\Psi^{-1}$. Local descriptions of forms in $\Omega^\grast(\lieA)$ were presented in \cite{Kuba96a}.

\section{Metrics and integration}
\label{sec-metric-integration}

The local descriptions of forms given in Proposition~\ref{prop-relationtrivforms} is used to define metrics on the transitive Lie algebroids $\lieA$ and then to define a natural notion of integration along the ``inner structure'' $\lieL$.

\subsection{Metrics}
\label{subsec-metrics}

\begin{definition}
Let $\lieA$ be a Lie algebroid over the manifold $\varM$. A metric on $\lieA$ is a symmetric, $C^\infty(\varM)$-linear map 
\begin{equation*}
\hg : \lieA \otimes_{C^\infty(\varM)} \lieA \rightarrow C^\infty(\varM)
\end{equation*}
\end{definition}

This definition does not require $\lieA$ to be transitive. But in the following, we shall study such a metric in the transitive case.

\begin{proposition}
\label{prop-constructionsaroundmetrics}
Let $\lieA \xrightarrow{\rho} \Gamma(T\varM)$ be a transitive Lie algebroid with kernel $\lieL$. 

A metric $\hg$ on $\lieA$ defines a metric $h = \iota^\ast \hg : \lieL \otimes_{C^\infty(\varM)} \lieL \rightarrow C^\infty(\varM)$ on the vector bundle $\varL$ for which $\lieL$ is the space of sections. Explicitly one has
\begin{equation*}
h( \gamma, \eta) = \hg(\iota(\gamma), \iota(\eta))
\end{equation*}
for any $\gamma, \eta \in \lieL$. We will call this metric the inner part of $\hg$.

Let $g$ be an ordinary metric on the manifold $\varM$. Then it defines a canonical metric $\hg = \rho^\ast g : \lieA \otimes_{C^\infty(\varM)} \lieA \rightarrow C^\infty(\varM)$ on $\lieA$ by the relation
\begin{equation*}
\hg(\kX, \kY) = g (\rho(\kX), \rho(\kY))
\end{equation*}
for any $\kX, \kY \in \lieA$.

Let $h : \lieL \otimes_{C^\infty(\varM)} \lieL \rightarrow C^\infty(\varM)$ be a metric on $\varL$ and let $\nabla$ be a connection on $\lieA$. Denote by $\lfc^\nabla \in \Omega^1(\lieA, \lieL)$ its associated connection $1$-form. Then the couple $(h, \nabla)$ defines a metric $\hg = {\lfc^\nabla}^\ast h : \lieA \otimes_{C^\infty(\varM)} \lieA \rightarrow C^\infty(\varM)$ on $\lieA$ by the relation
\begin{equation*}
\hg(\kX, \kY) = h (\lfc^\nabla(\kX), \lfc^\nabla(\kY))
\end{equation*}
This metric satisfies $h = \iota^\ast \hg$.
\end{proposition}

\begin{proof}
These claims are just direct applications of the definitions and the properties of the objects involved in the relations.
\end{proof}

The metric $\hg = \rho^\ast g$ vanishes on $\iota(\lieL)$ and the metric $\hg = {\lfc^\nabla}^\ast h$ vanishes on the image of $\nabla$ in $\lieA$. In the following we will introduce a kind of notion of non degeneracy in order to get rid of such metrics.

\begin{definition}
A metric $\hg$ on $\lieA$ is inner non degenerate if its inner metric $h = \iota^\ast \hg$ is non degenerate on $\lieL$, \textit{i.e.} if it is non degenerate as a metric on $\varL$.
\end{definition}

The constructions given in Proposition~\ref{prop-constructionsaroundmetrics} help us to decompose any metric on $\lieA$ into ``smaller'' entities. 

\begin{proposition}
\label{prop-connectionassociatedtoinnernondegeneratemetric}
Let $\hg$ be a inner non degenerate metric on $\lieA$. Then there exists a unique connection $\nabla^{\hg}$ on $\lieA$ such that, for any $X \in \Gamma(T\varM)$ and any $\gamma \in \lieL$, 
\begin{equation}
\label{eq-blockdiagmetricconnection}
\hg(\nabla^\hg_X, \iota(\gamma)) = 0
\end{equation}
\end{proposition}

\begin{proof}
Let $(U, \Psi, \nabla^{0})$ be a local trivialization of $\lieA$. Then one defines
\begin{equation*}
\hg_\loc(X \oplus \gamma, Y \oplus \eta) = \hg ( S(X \oplus \gamma), S(Y \oplus \eta))
\end{equation*}
for any $X \oplus \gamma, Y \oplus \eta \in \tla(U, \kg)$. One can introduce the following components of $\hg_\loc$:
\begin{align*}
g_\loc(X,Y) &= \hg_\loc(X, Y),
&
\hg^{\text{mix}}_\loc(X, \eta) &= \hg_\loc(X, \eta),
&
h_\loc(\gamma, \eta) &= \hg_\loc(\gamma, \eta),
\end{align*}
so that one has the decomposition
\begin{equation*}
\hg_\loc(X \oplus \gamma, Y \oplus \eta) = g_\loc(X,Y) + \hg^{\text{mix}}_\loc(X, \eta) + \hg^{\text{mix}}_\loc(Y, \gamma) + h(\gamma, \eta).
\end{equation*}

Using obvious notation, on $U_i \cap U_j \neq \ensvide$ one can easily establish the relations:
\begin{align*}
g_{\loc\, j}(X,Y) &= g_{\loc\, i}(X,Y) + \hg^{\text{mix}}_{\loc\, i}(X, \chi_{ij}(Y)) + \hg^{\text{mix}}_{\loc\, i}(Y, \chi_{ij}(X)) \\
& \hspace{15em}+ h_{\loc\, i}(\chi_{ij}(X), \chi_{ij}(Y))
\\[2mm]
\hg^{\text{mix}}_{\loc\, j}(X, \eta^j) &= \hg^{\text{mix}}_{\loc\, i}(X, \alpha^{i}_{j}(\eta^j)) + h_{\loc\, i}(\chi_{ij}(X), \alpha^{i}_{j}(\eta^j))
\\[2mm]
h_{\loc\, j}(\gamma^j, \eta^j) &= h_{\loc\, i}(\alpha^{i}_{j}(\gamma^j), \alpha^{i}_{j}(\eta^j)).
\end{align*}
Notice that the last line of these relations expresses the gluing relations for the local expressions of the inner metric $h = \iota^\ast \hg$.

Suppose the connection $\nabla^\hg$ on $\lieA$ which solves \eqref{eq-blockdiagmetricconnection} exists. It can be given locally by a family of $1$-forms $A_i \in \Omega^1(U_i)\otimes \kg$ which should satisfy \eqref{eq-gluingrelationlocalexpressionconnection}. These $1$-forms must solve the local expressions of the relation $\hg(\nabla^\hg_X, \iota(\gamma)) = 0$, which are:
\begin{equation*}
0 = \hg_{\loc\, i}(X \oplus A_i(X), 0 \oplus \gamma^i) = \hg^{\text{mix}}_{\loc\, i}(X, \gamma^i) + h_{\loc\, i}(A_i(X), \gamma^i).
\end{equation*}
Since $h_{\loc\, i}$ is non degenerate, on every $U_i$ there is a local $1$-form $A_i \in \Omega^1(U_i)\otimes \kg$ such that $h_{\loc\, i}(A_i(X), \gamma^i) = - \hg^{\text{mix}}_{\loc\, i}(X, \gamma^i)$. 

Let us now take a look at what happens on $U_i \cap U_j \neq \ensvide$. First, in this situation, the $\gamma^i$'s define a global object in $\lieL$, so that $\gamma^i = \alpha^{i}_{j}(\gamma^j)$. Then we have
\begin{align*}
h_{\loc\, i}(\alpha^{i}_{j} \circ A_j(X), \gamma^i) 
&= h_{\loc\, i}(\alpha^{i}_{j} \circ A_j(X), \alpha^{i}_{j}(\gamma^j))
= h_{\loc\, j}(A_j(X), \gamma^j)
\\
&= - \hg^{\text{mix}}_{\loc\, j}(X, \gamma^j)
\\
&= - \hg^{\text{mix}}_{\loc\, i}(X, \alpha^{i}_{j}(\gamma^j)) - h_{\loc\, i}(\chi_{ij}(X), \alpha^{i}_{j}(\gamma^j))
\\
&= - \hg^{\text{mix}}_{\loc\, i}(X, \gamma^i) - h_{\loc\, i}(\chi_{ij}(X), \gamma^i)
\\
&= h_{\loc\, i}(A_i(X), \gamma^i) - h_{\loc\, i}(\chi_{ij}(X), \gamma^i).
\end{align*}
Using again the non degeneracy of $h_{\loc\, i}$, one gets the relation $A_i = \alpha^{i}_{j} \circ A_j + \chi_{ij}$ which is \eqref{eq-gluingrelationlocalexpressionconnection}. The connection $\nabla^\hg$ then exists.

If $\nabla$ is another connection with the same property then for any $X \in \Gamma(T\varM)$ one has $\nabla^\hg_X - \nabla_X \in \iota(\lieL)$. This means that there exists $\gamma(X) \in \lieL$ such that $\nabla^\hg_X - \nabla_X = \iota\circ \gamma(X)$. Then one has $0 = \hg( \nabla^\hg_X - \nabla_X, \iota(\eta)) = \iota^\ast\hg(\gamma(X), \eta)$ for any $\eta \in \lieL$. Since the inner metric is non degenerate then $\gamma(X) = 0$. The connection is then unique.
\end{proof}

\begin{proposition}
\label{prop-tripleforinnernondegeneratemetric}
An inner non degenerate metric $\hg$ on $\lieA$ is equivalent to a triple $(g, h, \nabla)$ where $g$ is a metric on $\varM$, $h$ is a non degenerate metric on $\lieL$ and $\nabla$ is a connection on $\lieA$. 
The metric $\hg$ and the triple $(g, h, \nabla)$ are related by:
\begin{equation}
\label{eq-hatggha}
\hg(\kX, \kY) = g(\rho(\kX), \rho(\kY)) + h( \lfc^\nabla(\kX), \lfc^\nabla(\kY))
\end{equation}
where $\lfc^\nabla$ is the connection $1$-form associated to $\nabla$.
\end{proposition}

\begin{proof}
It is obvious that such a triple defines a inner non degenerate metric $\hg$ by the proposed relation.

In the opposite direction, Proposition~\ref{prop-connectionassociatedtoinnernondegeneratemetric} defines a unique connection $\nabla$ associated to $\hg$ satisfying \eqref{eq-blockdiagmetricconnection}. Using $\kX = \nabla_X - \iota \circ \lfc^\nabla(\kX)$ with $X = \rho(\kX)$ and $\lfc^\nabla$ the connection $1$-form associated to $\nabla$, one has
\begin{multline*}
\hg(\kX, \kY) =\\
 \hg(\nabla_X, \nabla_Y) - \hg(\iota \circ \lfc^\nabla(\kX), \nabla_Y) - \hg(\nabla_X, \iota \circ \lfc^\nabla(\kY)) + \hg(\iota \circ \lfc^\nabla(\kX), \iota \circ \lfc^\nabla(\kY)).
\end{multline*}
The two terms in the middle vanish by construction of $\nabla$. Define now
\begin{align*}
g(X,Y) &= \hg(\nabla_X, \nabla_Y),
&
h(\gamma, \eta) &= \hg(\iota(\gamma), \iota(\eta)).
\end{align*}
The triple $(g, h, \nabla)$ satisfies the requirements. Notice that the inner metric $h$ in this construction is exactly $h = \iota^\ast \hg$.
\end{proof}

\subsection{Mixed local basis of forms}

Let $\{E_a\}_{1 \leq a \leq n}$ be a basis of the $n$-dimensional Lie algebra $\kg$, and let $\{\theta^a\}_{1 \leq a \leq n}$ be its dual basis. Let $(U, \Psi, \nabla^{0})$ be a local trivialization of $\lieA$. Let $\nabla$ be a connection on $\lieA$. Then its connection $1$-form $\lfc^\nabla$ has a local expression $\lfc^\nabla_\loc = (A^a - \theta^a) E_a$ (summation over $a$ is understood), where $A \in \Omega^1(U)\otimes \kg$ has been defined before. Let us introduce the notation 
\begin{equation*}
\lfc^a = A^a - \theta^a \in \Omega^1_\tla(U).
\end{equation*}

\begin{definition}
\label{def-mixedbasis}
The local $1$-forms $\lfc^a$ on $U$ are called the mixed basis on the inner part of $\Omega^1_\tla(U)$ relative to the connection $\nabla$ and to the basis $\{E_a\}_{1 \leq a \leq n}$ of $\kg$.
\end{definition}

Let $\omega \in \Omega^q(\lieA, \lieL)$ and denote by $\omega_{\loc} \in \Omega^q_\tla(U,\kg)$ its trivialization over $U$. Then one has
\begin{equation*}
\omega_{\loc} = \sum_{r+s=q} \omega^{\theta\; (r,s)}_{\mu_1 \ldots \mu_r a_1 \ldots a_s} \dd x^{\mu_1} \ordwedge \cdots \ordwedge \dd x^{\mu_r} \ordwedge \theta^{a_1} \ordwedge \cdots \ordwedge \theta^{a_s}
\end{equation*}
with $\omega^{\theta}_{\mu_1 \ldots \mu_r a_1 \ldots a_s} : U \rightarrow \kg$.
Using $\theta^a = A^a - \lfc^a$, this expression can be written as
\begin{equation*}
\omega_{\loc} = \sum_{r+s=q} \omega^{(r,s)}_{\mu_1 \ldots \mu_r a_1 \ldots a_s} \dd x^{\mu_1} \ordwedge \cdots \ordwedge \dd x^{\mu_r} \ordwedge \lfc^{a_1} \ordwedge \cdots \ordwedge \lfc^{a_s}
\end{equation*}
for some new components $\omega^{(r,s)}_{\mu_1 \ldots \mu_r a_1 \ldots a_s} : U \rightarrow \kg$ which are polynomials in the $A^a_\mu$'s for $A^a = A^a_\mu \dd x^{\mu}$.

We shall be now interested in the gluing rules of these components according to two trivializations of $\lieA$.

\begin{proposition}
Let us introduce the matrix valued functions $G^{i}_{j}= \left({G^{i}_{j}}^b_a\right)_{1 \leq a,b \leq n}$ on $U_{ij}=U_i \cap U_j \neq \ensvide$ defined by $\alpha^{i}_{j}(E_a) = {G^{i}_{j}}^b_a E_b$ (summation over $b$). With obvious notations, on $U_{ij}$ one has 
\begin{equation*}
\lfc_i^a = {G^{i}_{j}}^a_b \lfc_j^b \circ s_{i}^{j}
\end{equation*}
where $s_{i}^{j} = S_j^{-1} \circ S_i : \tla(U_{ij}, \kg) \xrightarrow{\simeq} \tla(U_{ij}, \kg)$.
\end{proposition}

\begin{proof}
This is a direct consequence of \eqref{eq-changelocaltrivializationforms}.
\end{proof}

Notice that this relation can also be written as 
\begin{equation}
\label{eq-relationtriva}
\halpha_{i}^{\,j} (\lfc_i^a) = {G^{i}_{j}}^a_b \lfc_j^b
\end{equation}
where $\halpha_{i}^{\,j}$ is the one defined on forms in $\Omega^\grast_\tla(U_{ij})$. 

On $U_{ij} = U_i \cap U_j \neq \ensvide$, with obvious notations, we decompose $\omega^i_{\loc}$ along the $\lfc_i^a$'s and $\omega^j_{\loc}$ along the $\lfc_j^a$'s. Using \eqref{eq-relationtriva}, $\halpha_{j}^{\,i}(\omega^j_{\loc}) = \omega^i_{\loc}$, $\halpha_{j}^{\,i}(\omega^j_{\mu_1 \ldots \mu_r a_1 \ldots a_s}) = \alpha_{j}^{i}(\omega^j_{\mu_1 \ldots \mu_r a_1 \ldots a_s})$ and $\halpha_{j}^{\,i}(\dd x^\mu) = \dd x^\mu$, one gets
\begin{equation}
\label{eq-relationtrivcomponentsona}
\omega^i_{\mu_1 \ldots \mu_r a_1 \ldots a_s} = {G^{j}_{i}}^{b_1}_{a_1} \cdots {G^{j}_{i}}^{b_s}_{a_s} \alpha_{j}^{i}(\omega^j_{\mu_1 \ldots \mu_r b_1 \ldots b_s})
\end{equation}
These homogeneous gluing relations motivates the decomposition of global forms on $\lieA$ along the $\lfc^a$ instead of the $\theta^a$'s.

Let $(g,h,\nabla)$ be a triple as in Proposition~\ref{prop-tripleforinnernondegeneratemetric}. Let $(U, \Psi, \nabla^{0})$ be a local trivialization of $\lieA$. Denote by $\lfc^a \in \Omega^1_\tla(U)$ the mixed basis relative to both $\nabla$ and the basis $\{E_a\}_{1 \leq a \leq n}$ of $\kg$. Let us assume that $U$ is the support of a chart of $\varM$, with coordinates $(x^\mu)$. 
Locally, we can write \eqref{eq-hatggha} as
\begin{equation*}
\hg_\loc = g_{\mu \nu} \dd x^\mu \dd x^\nu + h_{\loc\, a, b} \lfc^a \lfc^b
\end{equation*}
where $g_{\mu \nu}$ are the components of the metric $g$ on $T\varM$ in a local chart over $U$ and where $h_{\loc\, a, b} = h_\loc(E_a, E_b)$. The mixed basis permits to diagonalize by blocks the local expression of the metric $\hg$.

\subsection{Inner orientation and integration}
\label{subsec-innerorientationandintegration}

Let as before $\{E_a\}_{1 \leq a \leq n}$ denotes a basis of the $n$-dimensional Lie algebra $\kg$ and $\{\theta^a\}_{1 \leq a \leq n}$ its dual basis. 

An inner metric $h$ on $\lieL$ induces a local metric $h_\loc$ on $U \times \kg$ given by $h_\loc(\gamma, \eta) = h(\Psi(\gamma), \Psi(\eta))$ for any $\gamma, \eta \in \Gamma(U \times \kg)$. This is the same as in the proof of Proposition~\ref{prop-connectionassociatedtoinnernondegeneratemetric}. Let us define $h_{\loc\, a, b} = h_\loc(E_a, E_b)$. 

On $U_i$, denote by $\gamma_i = \gamma_i^a E_a$ the local expression of an element $\gamma \in \lieL$. On $U_{ij} = U_i \cap U_j \neq \ensvide$, the relation $\gamma_i = \alpha^{i}_{j}(\gamma_j)$ induces the relation $\gamma_i^a = {G^{i}_{j}}^{a}_{b} \gamma_j^b$. For any $\gamma, \eta \in \lieL$, the compatibility relation $h^j_\loc(\gamma_j, \eta_j) = h^i_\loc(\gamma_i, \eta_i) = h^i_\loc(\alpha^{i}_{j}(\gamma_j), \alpha^{i}_{j}(\gamma_j))$ gives
\begin{equation}
\label{eq-relationtrivh}
h^j_{\loc\, {b_1}, {b_2}} = {G^{i}_{j}}^{a_1}_{b_1} {G^{i}_{j}}^{a_2}_{b_2} h^i_{\loc\, {a_1}, {a_2}}
\end{equation}

The vector bundle $\varL$ is orientable if $\det({G^{i}_{j}}) > 0$ for any $i,j$ such that $U_i \cap U_j \neq \ensvide$. We will say that $\lieL$ is orientable if $\varL$ is orientable.

\begin{definition}
We will say that a transitive Lie algebroid is inner orientable if its kernel is orientable.
\end{definition}

This notion of ``inner orientable Lie algebroid'' is the same as the notion of ``vertically orientable Lie algebroid'' used in \cite{MR1908998}.

\begin{proposition}
\label{prop-globalinnerform}
On each $U_i$, let $|h^i_\loc|$ denotes the absolute value of the determinant of the matrix $(h^i_{\loc\, a, b})_{a, b}$. If $\lieL$ is orientable then on $U_{ij} = U_i \cap U_j \neq \ensvide$ one has
\begin{equation*}
\halpha_{i}^{\,j}\left( \sqrt{|h^i_\loc|} \lfc_i^{1} \ordwedge \cdots \ordwedge \lfc_i^{n}\right) = \sqrt{|h^j_\loc|} \lfc_j^{1} \ordwedge \cdots \ordwedge \lfc_j^{n}
\end{equation*}
This implies that there exists a global form $\omega_{h,\lfc} \in \Omega^\grast(\lieA)$ of maximal inner degree $n$ defined locally by
\begin{equation*}
\omega_{h,\lfc} = (-1)^n \sqrt{|h_\loc|} \lfc^{1} \ordwedge \cdots \ordwedge \lfc^{n}
\end{equation*}
\end{proposition}

The form $\omega_{h,\lfc} \in \Omega^\grast(\lieA)$ will play the role of a ``volume form'' for fiber integration.

\begin{proof}
On one hand, we have
\begin{align*}
\halpha_{i}^{\,j}\left( \lfc_i^{1} \ordwedge \cdots \ordwedge \lfc_i^{n} \right) &= {G^{i}_{j}}^1_{b_1} \cdots {G^{i}_{j}}^n_{b_n} \lfc_j^{b_1} \ordwedge \cdots \ordwedge \lfc_j^{b_n}
\\
&= \sum_{b_i} \varepsilon^{b_1 \dots b_n} {G^{i}_{j}}^1_{b_1} \cdots {G^{i}_{j}}^n_{b_n} \lfc_j^{1} \ordwedge \cdots \ordwedge \lfc_j^{n}
\\
&= \det({G^{i}_{j}}) \lfc_j^{1} \ordwedge \cdots \ordwedge \lfc_j^{n}
\end{align*}
where $\varepsilon^{b_1 \dots b_n}$ is the totally antisymmetric Levi-Civita symbol.

On the other hand, a straightforward computation gives 
\begin{equation}
\det(h^j_\loc) = \det({G^{i}_{j}})^2 \det(h^i_\loc) \label{eq-relationtrivdeth}
\end{equation}
so that $|h^i_\loc| = |\det({G^{i}_{j}})|^{-2} |h^j_\loc|$ as a density.

When $\det({G^{i}_{j}}) > 0$, one has $\halpha_{i}^{\,j}\left( \sqrt{|h^i_\loc|} \lfc_i^{1} \ordwedge \cdots \ordwedge \lfc_i^{n} \right) = \sqrt{|h^j_\loc|} \lfc_j^{1} \ordwedge \cdots \ordwedge \lfc_j^{n}$.
\end{proof}

Any form $\omega \in \Omega^\grast(\lieA, \lieL)$ of maximal degree $n$ in the inner direction can be written locally on $U_i$ as 
\begin{equation*}
\omega^i_{\loc} = (-1)^n \omega^\maxinner_{\loc\, i} \sqrt{|h^i_\loc|} \lfc_i^{1} \ordwedge \cdots \ordwedge \lfc_i^{n} + \omega^R
= \omega^\maxinner_{\loc\, i}\, \omega_{h,\lfc} + \omega^R
\end{equation*}
where $\omega^R$ contains only terms of lower degrees in the $\lfc_i^a$'s, with $\omega^\maxinner_{\loc\, i} \in \Omega^\grast(U_i) \otimes \kg$ (``$\maxinner$'' stands for ``maximum inner'').
The factor $\omega^\maxinner_{\loc\, i}$ can also be defined as the factor of $\sqrt{|h^i_\loc|} \theta^1 \ordwedge \cdots \ordwedge \theta^n$ in $\omega^i_{\loc}$ and $\omega^R$ is a sum of terms with degree in the $\theta^a$'s less or equal to $n-1$.

Two such local expressions can be compared on intersecting trivializations. Applying $\halpha_{i}^{\,j}$ on $\omega^i_{\loc}$ and using Proposition~\ref{prop-globalinnerform}, one gets
\begin{equation*}
\alpha^{j}_{i}(\omega^\maxinner_{\loc\, i}) = \omega^\maxinner_{\loc\, j}
\end{equation*}
so that the forms $\omega^\maxinner_{\loc\, i}$ define a global form $\omega^\maxinner \in \Omega^{\grast-n}(\varM, \varL)$, the space of (de~Rham) forms on $\varM$ with values in the vector bundle in Lie algebras $\varL$.

\begin{definition}
\label{def-innerintegration}
On an inner orientable transitive Lie algebroid equipped with a metric, one defines the inner integration as the operation
\begin{align*}
\int_\inner : \Omega^\grast(\lieA,\lieL) &\rightarrow \Omega^{\grast-n}(\varM, \varL)
&
\omega &\mapsto \omega^\maxinner
\end{align*}
This inner integration is zero when applied to forms which does not contain terms of maximal inner degree $n$.

\end{definition}

Because the form $\omega^\maxinner$ defined to be the result of this inner integration is in fact the factor of $\sqrt{|h^i_\loc|} \theta^1 \ordwedge \cdots \ordwedge \theta^n$, this integration does not depend on the choice of the connection but only on the inner metric $h$.

The same construction yields an inner integration for $C^\infty(\varM)$-valued forms through
\begin{equation*}
\int_\inner : \Omega^\grast(\lieA) \rightarrow \Omega^{\grast-n}(\varM)
\end{equation*}
Notice that by construction $\int_\inner \omega_{h,\lfc} =1$ where $\omega_{h,\lfc} \in \Omega^\grast(\lieA)$ is the ``volume form'' defined in Proposition~\ref{prop-globalinnerform}.

The global form $\omega_{h,\lfc}$ plays a dual role to the non-singular cross section $\varepsilon \in \exter^n \lieL$ used in \cite{MR1908998} to define integration along the fiber on $\Omega^\grast(\lieA)$. Given a non degenerate inner metric $h$ on $\lieL$, one can define
\begin{equation}
\label{eq-epsilonkubarski}
\varepsilon_{\loc} = (-1)^n \sqrt{|h_\loc|}^{\;-1} E_1 \wedge \cdots \wedge E_n
\end{equation}
in any local trivialization $(U, \Psi, \nabla^{0})$ of $\lieA$. These local expressions define a global form $\varepsilon \in \Gamma(\exter^n \varL) = \exter^n \lieL$ which satisfies
\begin{equation*}
i_\varepsilon \omega_{h,\lfc} = 1
\end{equation*}
where the operation $i_\varepsilon$ on $\Omega^\grast(\lieA)$ is defined as in \cite{MR1908998} and corresponds there to the integral along the fiber on $\Omega^\grast(\lieA)$. This relates our constructions to the ones proposed by Kubarski. The present notion of inner integration is also a direct generalization for transitive Lie algebroids of the notion of ``noncommutative'' integration defined and studied in \cite{Mass15} (see also \cite{Mass30} for more constructions which are related to the present situation).

In order to define a global integration on forms, we suppose from now on that the manifold $\varM$ is orientable.

\begin{definition}
A transitive Lie algebroid is orientable if it is inner orientable and if its base manifold is orientable.
\end{definition}

Then we can define:
\begin{definition}
The integration on an orientable transitive Lie algebroid equipped with a metric is the composition of the inner integration on $\Omega^\grast(\lieA)$ with the integration of forms on the base manifold. This integration is denoted by
\begin{equation*}
\int_\lieA \omega = \int_\varM \int_\inner \omega \in \gC
\end{equation*}
for any $\omega \in \Omega^\grast(\lieA)$. 
\end{definition}

Obviously this definition makes sense only when the integral on $\varM$ converges. For instance, this is always the case when $\varM$ is compact or for compactly supported (relative to $\varM$) forms on $\lieA$.

This definition is the same as the one given in Definition~2.1 in \cite{MR1908998}. The present definition only consider forms in $\Omega^\grast(\lieA)$. It can be extended to $\Omega^\grast(\lieA, \lieL)$ in the case of the transitive Lie algebroid of derivations of a vector bundle using the ``extended'' inner integration $\int^{\tr}_\inner$ which will be defined in \eqref{eq-definnerintegrationtrace}.

This integral is non zero only if its contains a non zero term which is of maximal degree in both the inner direction and the spatial direction. In that particular case, the integral depends only on this term. This integration has several properties which have been described in \cite{MR1908998}.

\begin{definition}
\label{def-scalarproductformsfunctions}
Let $\lieA$ be an orientable transitive Lie algebroid equipped with a metric. For any $\omega \in \Omega^\grast(\lieA)$ and $\eta \in \Omega^\grast(\lieA)$, one defines their scalar product as
\begin{equation*}
\langle \omega, \eta \rangle = \int_\lieA \omega\, \eta \in \gC
\end{equation*}
\end{definition}

Obviously this scalar product is always zero if the graduation of the product $\omega\, \eta$ is not maximal in the inner directions and in the spatial directions.

This scalar product is considered in \cite{MR1908998} to study Poincaré duality on $\Omega^\grast(\lieA)$.

Let $h$ be an inner metric on $\lieL$. We can extend it to a map $h : \Omega^p(\lieA, \lieL) \otimes_{C^\infty(\varM)} \Omega^q(\lieA, \lieL) \rightarrow \Omega^{p+q}(\lieA)$ by
\begin{multline*}
h(\omega, \eta)(\kX_1, \dots, \kX_{p+q}) =\\
 \frac{1}{p!q!} \sum_{\sigma\in \kS_{p+q}} (-1)^{\sign(\sigma)} h(\omega(\kX_{\sigma(1)}, \dots, \kX_{\sigma(p)}), \eta(\kX_{\sigma(p+1)}, \dots, \kX_{\sigma(p+q)}))
\end{multline*}
for any $\omega \in \Omega^p(\lieA, \lieL)$ and $\eta \in \Omega^q(\lieA, \lieL)$. For $p=q=0$, this is the original map $h$. Notice that $h(\omega, \eta) = (-1)^{pq} h(\eta, \omega)$.

\begin{definition}
\label{def-scalarproductforms}
Let $\lieA$ be an orientable transitive Lie algebroid equipped with a metric. For any $\omega \in \Omega^\grast(\lieA, \lieL)$ and $\eta \in \Omega^\grast(\lieA, \lieL)$, one defines their scalar product as
\begin{equation*}
\langle \omega, \eta \rangle = \int_\lieA h(\omega, \eta) \in \gC
\end{equation*}
\end{definition}

\subsection{Hodge operators}
\label{subsec-hodgeoperators}

In the following, we suppose that $\lieA$ is an orientable transitive Lie algebroid equipped with a metric $\hg = (g,h,\nabla)$.

Let $\omega \in \Omega^p(\lieA, \lieL)$ be written locally in a trivialization $(U, \Psi, \nabla^{0})$ of $\lieA$ as 
\begin{equation*}
\omega_{\loc} = \sum_{r+s=p} \omega^{(r,s)}_{\mu_1 \ldots \mu_r a_1 \ldots a_s} \dd x^{\mu_1} \ordwedge \cdots \ordwedge \dd x^{\mu_r} \ordwedge \lfc^{a_1} \ordwedge \cdots \ordwedge \lfc^{a_s}
\end{equation*}
where the $\lfc^{a}$'s are the components of the local expression of the connection $1$-form associated to $\nabla$ and where $\omega^{(r,s)}_{\mu_1 \ldots \mu_r a_1 \ldots a_s}$ are functions on $U$ with values in $\kg$.

Consider the form in $\Omega^{m+n-p}_\tla(U,\kg)$ defined by
\begin{multline*}
\hodgeast \omega_{\loc} = \sum_{r+s=p} (-1)^{s(m-r)}\; \frac{1}{r! s!}\; \sqrt{|h_\loc|} \sqrt{|g|}\; \omega^{(r,s)}_{\mu_1 \ldots \mu_r a_1 \ldots a_s}\; \epsilon_{\nu_1 \ldots \nu_m}\; \epsilon_{b_1 \ldots b_n} \\
\times g^{\mu_1 \nu_1} \cdots g^{\mu_r \nu_r}\; h_\loc^{a_1 b_1} \cdots h_\loc^{a_s b_s}\; \dd x^{\nu_{r+1}} \ordwedge \cdots \ordwedge \dd x^{\nu_m} \ordwedge \lfc^{b_{s+1}} \ordwedge \cdots \ordwedge \lfc^{b_n}
\end{multline*}
where $\epsilon_{\nu_1 \ldots \nu_m}$ and $\epsilon_{b_1 \ldots b_n}$ are the totally antisymmetric Levi-Civita symbols.

Using \eqref{eq-relationtriva}, \eqref{eq-relationtrivcomponentsona}, \eqref{eq-relationtrivh}, \eqref{eq-relationtrivdeth} one can establish that $\halpha_{j}^{\,i}(\hodgeast \omega_{\loc}^j) = \hodgeast \omega_{\loc}^i$ so that, by Proposition~\ref{prop-relationtrivforms}, $\hodgeast \omega \in \Omega^{m+n-p}(\lieA, \lieL)$ is well-defined. 

\begin{definition}
The map $\hodgeast : \Omega^p(\lieA, \lieL) \rightarrow \Omega^{m+n-p}(\lieA, \lieL)$ is the Hodge star operation on the orientable transitive Lie algebroid $\lieA$ associated to the metric $\hg$.
\end{definition}

\begin{proposition}
For any $\omega \in \Omega^p(\lieA, \lieL)$ one has
\begin{equation*}
\hodgeast \hodgeast \omega = (-1)^{(m+n-p)p} \omega
\end{equation*}
\end{proposition}

\begin{proof}
This is just direct a computation using the definition of $\hodgeast$ and some combinatoric properties of the Levi-Civita symbols.
\end{proof}

This Hodge star operation defines a natural scalar product on any $\Omega^p(\lieA, \lieL)$ by
\begin{equation*}
( \omega, \eta ) = \langle \omega, \hodgeast \eta  \rangle
\end{equation*}
for any $\omega, \eta  \in \Omega^p(\lieA, \lieL)$ where $\langle -, - \rangle$ is defined in Definition~\ref{def-scalarproductforms}.

\begin{proposition}
For any $\omega, \eta  \in \Omega^p(\lieA, \lieL)$ written in a trivialization of $\lieA$ as
\begin{align*}
\omega_{\loc} &= \sum_{r+s=p} \omega^{(r,s)\, a}_{\mu_1 \ldots \mu_r a_1 \ldots a_s} \dd x^{\mu_1} \ordwedge \cdots \ordwedge \dd x^{\mu_r} \ordwedge \lfc^{a_1} \ordwedge \cdots \ordwedge \lfc^{a_s} \otimes E_a
\\
\eta_{\loc} &= \sum_{r+s=p} \eta^{(r,s)\, b}_{\nu_1 \ldots \nu_r b_1 \ldots b_s} \dd x^{\nu_1} \ordwedge \cdots \ordwedge \dd x^{\nu_r} \ordwedge \lfc^{b_1} \ordwedge \cdots \ordwedge \lfc^{b_s} \otimes E_b
\end{align*}
one has
\begin{equation*}
( \omega, \eta ) =  (-1)^n\sum_{r+s=p} (-1)^{s(m-r)}\; (m-r)!\; (n-s)!\; \omega^{(r,s)\, a}_{\mu_1 \ldots \mu_r a_1 \ldots a_s}\; \eta_{(r,s)\, a}^{\mu_1 \ldots \mu_r a_1 \ldots a_s}
\end{equation*}
with
\begin{equation*}
\eta_{(r,s)\, a}^{\mu_1 \ldots \mu_r a_1 \ldots a_s} = g^{\mu_1 \nu_1} \cdots g^{\mu_r \nu_r}\; h^{a_1 b_1} \cdots h^{a_s b_s} h_{ab}\; \eta^{(r,s)\, b}_{\nu_1 \ldots \nu_r b_1 \ldots b_s}
\end{equation*}
where $(g^{\mu \nu})$ and $(h^{a b})$ are the inverse matrices of $(g_{\mu \nu})$ and $(h_{a b})$ respectively.
\end{proposition}

\begin{proof}
This is just a combinatoric straightforward computation.
\end{proof}

Notice that the Hodge star operation $\hodgeast$ is also well defined on $\Omega^\grast(\lieA)$ where it permits to introduce a scalar product $( \omega, \eta ) = \langle \omega, \hodgeast \eta  \rangle$ using Definition~\ref{def-scalarproductformsfunctions}. A similar relation as the one given in the previous proposition can be established.

\section{Applications to cohomological structures}
\label{sec-cohomology}

Cohomological properties in the context of Lie algebroids have been studied and used, for instance in order to control the deformation theory \cite{MR2443928} or to define characteristic classes \cite{Fern02a}. 

Some spectral arguments are developed in \cite{Mack05a} to relate differential structures at the level of the base manifold $\varM$ to the cohomology of the Lie algebroid with values in a representation space.

In this section, we apply the local description of generalized forms to construct a  bicomplex which generalizes the ordinary Čech-de~Rham bicomplex. Using spectral sequence argument, this construction relates the computation of the cohomology of the complex $\Omega^\grast(\lieA, \lieL)$ to the computation of a Čech cohomology in a convenient presheaf.

The construction presented here is inspired by the ``particular case'' considered in the noncommutative framework in \cite{Mass15}. It also generalizes a similar construction performed in \cite{MR2020382} for the differential complex $\Omega^\grast(\lieA)$. 

The map $\halpha_{j}^{\,i}$ defined in \eqref{eq-definitionalphahat} is the key ingredient of the following considerations. It permits to introduce a restriction map on the natural candidate for a presheaf associated to $\Omega^\grast(\lieA, \lieL)$.

\smallskip
Let $(U, \Psi_U, \nabla^{0,U})$ and $(V, \Psi_V, \nabla^{0,V})$ be local trivializations of $\lieA$ such that $V \subset U$. For any $\omega \in \Omega^\grast_\tla(U,\kg)$, denote by $\omega_{\upharpoonright V}$ its restriction to $V$, which is the restriction of the de~Rham part of the form considered as a element in the bicomplex $\Omega^\grast(U) \otimes \exter^\grast \kg^\ast \otimes \kg$.

Denote by $\halpha_U^V$ the map defined in \eqref{eq-definitionalphahat} associated to these two trivializations over the open subset $U\cap V = V$. Look at $\halpha_U^V$ as a map between forms in $\Omega^\grast_\tla(U\cap V,\kg)$ ``trivialized'' for $(U, \Psi_U, \nabla^{0,U})$ and forms in $\Omega^\grast_\tla(V,\kg)$ ``trivialized'' for $(V, \Psi_V, \nabla^{0,V})$. Actually, forms in $\Omega^\grast_\tla(V,\kg)$ are not considered anymore as trivializations of global forms on $\varM$.

As a direct consequence of Proposition~\ref{prop-localdifferentialcommutes}, the map
\begin{equation}
\label{eq-defiUV}
i_U^V : \Omega^\grast_\tla(U,\kg) \rightarrow \Omega^\grast_\tla(V,\kg),
\qquad
i_U^V(\omega) = \halpha_U^V(\omega_{\upharpoonright V})
\end{equation}
is a morphism of graded differential Lie algebras. For $W \subset V \subset U$, one has the chain relation $i_V^W \circ i_U^V = i_U^W$. Equipped with this restriction map, $U \mapsto \Omega^\grast_\tla(U,\kg)$ is a presheaf of graded differential Lie algebras, which we shall denote by $\caF$.

Let $\kU = \{ U_{i} \}_{i \in I}$ be a good cover of $\varM$ indexed by an ordered set $I$ such that there is a Lie algebroid atlas $\{(U_i, \Psi_i, \nabla^{0,i})\}_{i \in I}$ on $\lieA$ subordinated to it. 

On any $U_{i_0\cdots i_p} = U_{i_0} \cap \cdots \cap U_{i_p} \neq \ensvide$, we suppose given once and for all a trivialization of $\lieA_{|U_{i_0\cdots i_p}}$. For instance, a trivialization on $U_{i_0\cdots i_p} \subset U_{i_0}$ could be the restriction of the trivialization associated to $U_{i_0}$.

We now define an algebraic version of the ordinary Čech-de~Rham bicomplex \cite{BottTu95} associated to the presheaf $\caF$. For $p\geq 0$ and $q \geq 0$, consider the bicomplex
\begin{equation*}
C^{p,q}(\kU, \caF) = \prod_{i_0 < \cdots < i_p} \Omega^q_\tla(U_{i_0\cdots i_p},\kg)
\end{equation*}
We denote by $\omega = (\omega_{i_0\cdots i_p}) \in C^{p,q}(\kU, \caF)$ an homogeneous element of this bicomplex, with $\omega_{i_0\cdots i_p} \in  \Omega^q_\tla(U_{i_0\cdots i_p},\kg)$.

\begin{definition}
With the previous notations, we define
\begin{equation*}
\delta : C^{p,q}(\kU, \caF) \rightarrow C^{p+1,q}(\kU, \caF)
\end{equation*}
by
\begin{equation*}
(\delta \omega)_{i_0\cdots i_{p+1}} = \sum_{k=0}^{p} (-1)^{k+1} i_{U_{i_0\cdots \remove{i_k} \cdots i_{p+1}}}^{U_{i_0\cdots i_{p+1}}}(\omega_{i_0\cdots \remove{i_k} \cdots i_{p+1}})
\end{equation*}
where $i_0\cdots \remove{i_k} \cdots i_{p+1}$ means the omission of the index $i_k$.
\end{definition}

One can extend the bicomplex $C^{p,\grast}(\kU, \caF)$ in order to incorporate the case $p=-1$ by defining $C^{-1,q}(\kU, \caF) = \Omega^q(\lieA, \lieL)$. Then $\delta : C^{-1,q}(\kU, \caF) \rightarrow C^{0,q}(\kU, \caF)$ is defined as the local trivialization of forms in the atlas $\{(U_i, \Psi_i, \nabla^{0,i})\}_{i \in I}$.

\begin{lemma}
One has $\delta^2=0$, $\delta \hd_\tla = \hd_\tla \delta$, and for any $q\geq 0$ the sequence
\begin{equation*}
\xymatrix@1@C=15pt{{\algzero} \ar[r] & {C^{-1,q}(\kU, \caF)} \ar[r]^-{\delta} & {C^{0,q}(\kU, \caF)} \ar[r]^-{\delta} & {C^{1,q}(\kU, \caF)} \ar[r]^-{\delta} & {C^{2,q}(\kU, \caF)} \ar[r]^-{\delta} & {\cdots}}
\end{equation*}
is exact.
\end{lemma}

\begin{proof}
The relation $\delta^2=0$ is a consequence of the chain relation on $i_U^V$.

The Čech differential $\delta$ commutes with $\hd_\tla$ because the restriction maps $i_U^V$ commute with $\hd_\tla$.

In order to prove the exactness, we introduce a partition of unity $\{\rho_i\}_{i \in I}$ subordinated to the good cover $\kU = \{ U_{i} \}_{i \in I}$.

Let $\omega \in C^{p,q}(\kU, \caF)$ such that $\delta \omega = 0$. Then for any $i_0 < \cdots < i_p$, let us define
\begin{equation*}
\tau_{i_0 \ldots i_p} = \sum_{k} (\halpha_{i_0 \ldots i_p}^{k i_0 \ldots i_p})^{-1} \left( \rho_k \omega_{k i_0 \ldots i_p} \right) 
\end{equation*}
where $\rho_k \omega_{k i_0 \ldots i_p}$ is extended on $U_{i_0 \ldots i_p}$ by $0$ outside of $U_{k i_0 \ldots i_p} \subset U_{i_0 \ldots i_p}$. In this formula, we use the simplified notation $\halpha_{i_0 \ldots i_p}^{k i_0 \ldots i_p} = \halpha_{U_{i_0 \ldots i_p}}^{U_{k i_0 \ldots i_p}}$ and in the second summation we use the Čech convention $\omega_{\ldots i j \ldots} = - \omega_{\ldots j i \ldots}$. 

A straightforward computation shows that $\delta \tau = \omega$.
\end{proof}

As a direct consequence of this Lemma, using standard arguments in bicomplexes, one has:

\begin{proposition}
The cohomology of $(\Omega^\grast(\lieA, \lieL), \hd)$ is the cohomology of the total complex of the bicomplex $(C^{\grast,\grast}(\kU, \caF), \hd_\tla, \delta)$.
\end{proposition}

To the bicomplex $(C^{\grast,\grast}(\kU, \caF), \hd_\tla, \delta)$ one can associate a spectral sequence  $\{ E_r, \dd_r\}_{r \geq 0}$ induced by the natural filtration along $p$:
\begin{equation*}
F^p C^{\grast,\grast}(\kU, \caF) = \bigoplus_{k \geq p} \bigoplus_{q \geq 0} C^{k,q}(\kU, \caF)
\end{equation*}
The first term of this spectral sequence is
\begin{equation*}
E_1^{p,q} = H^{p,q}(C^{\grast,\grast}(\kU, \caF), \hd_\tla) = C^p(\kU, \caH^q)
\end{equation*}
where $\caH^q$ is the presheaf which associates to an open subset $U \subset \varM$ the cohomology space $\caH^q(U) = H^q(\Omega^\grast_\tla(U,\kg), \hd_\tla)$.

The second term of this spectral sequence is then
\begin{equation*}
E_2^{p,q} = H^p(\varM; \caH^q)
\end{equation*}
This is a Leray-Serre spectral sequence \cite{McCl00} for which the ``fibration'' is not just along ordinary spaces but along ``differential structures''.

Because $\kU$ is a good cover, one has
\begin{equation*}
H^q(\Omega^\grast_\tla(U_{i_0 \ldots i_p},\kg), \hd_\tla) = H^q(\kg; \kg),
\end{equation*}
the Lie algebra cohomology of the usual differential complex $(\kg \otimes \exter^\grast \kg^\ast ,\ds')$.

Remember that the restriction map $i_U^V$ for the presheaf $\caF$ makes use of the action of $\halpha_U^V$. The induced restriction map for the presheaf $\caH^q$ is obtained by the induced action of $\halpha_U^V$ in cohomology. This implies that the presheaf $\caH^q$ is not necessarily a constant presheaf.

\medskip
In \cite{Mack05a}, a spectral sequence based on the action of $\lieL$ is constructed to give cohomological information on a transitive Lie algebroid $\lieA$. It is based on a differential calculus similar to the one we use here except that it takes its values in a representation space of the Lie algebroid. These arguments can be adapted easily to the present context.

In Section~\ref{sec-applicationAtiyah} we shall use a similar argument to present another spectral sequence. Our construction is based on a Cartan operation of a finite dimensional Lie algebra \cite{Mass38}, which is more manageable that the infinite dimensional Lie algebra $\lieL$.

\section{Application to Atiyah Lie algebroids}
\label{sec-applicationAtiyah}

Let $\varP \xrightarrow{\pi} \varM$ be a $G$-principal bundle over $\varM$. In this section, we suppose that the group $G$ is connected. Denote by $\raR_g : \varP \rightarrow \varP$, $\raR_g(p) = p \cdotaction g$, the right action of $G$ on $\varP$ and by $\kg$ the Lie algebra of the Lie group $G$.
Define the two spaces
\begin{align*}
\Gamma_G(T\varP) &= \{ \sfX \in \Gamma(T\varP) \, / \, \raR_{g\,\ast}\sfX = \sfX \text{ for all } g \in G \}
\\
\Gamma_G(\varP, \kg) &= \{ v : P \rightarrow \kg \, / \, v(p \cdotaction g) = \Ad_{g^{-1}} v(p) \text{ for all } g \in G \}.
\end{align*}
The Lie bracket on $\Gamma(T\varP)$ restricts to $\Gamma_G(T\varP)$ and $\kg$ induces a Lie bracket on $\Gamma(\varP, \kg)$. These two spaces are $C^\infty(\varM)$-modules when one considers $f \in C^\infty(\varM)$ as the $\raR$-invariant function $\pi^\ast f \in C^\infty(\varP)$.
The map
\begin{equation*}
\iota(v)(p) = -v(p)^\varP_{|p} = \left( \frac{d}{dt} p \cdotaction e^{-t v(p)} \right)_{|t=0}
\end{equation*}
defines an inclusion $\iota : \Gamma_G(\varP, \kg) \rightarrow \Gamma_G(T\varP)$ of Lie algebras and $C^\infty(\varM)$-modules.

\begin{definition}
The short exact sequence of Lie algebras and $C^\infty(\varM)$-modules
\begin{equation*}
\xymatrix@1{{\algzero} \ar[r] & {\Gamma_G(\varP, \kg)} \ar[r]^-{\iota} & {\Gamma_G(T\varP)} \ar[r]^-{\pi_\ast} & {\Gamma(T \varM)} \ar[r] & {\algzero}}
\end{equation*}
defines $\Gamma_G(T\varP)$ as a transitive Lie algebroid over $\varM$, with kernel $\lieL = \Gamma_G(\varP, \kg) = \Gamma(\caL)$ where $\caL = \varP \times_\Ad \kg$ is the associated vector bundle. This is the Atiyah Lie algebroid associated to $\varP$ \cite{MR0086359}.
\end{definition}

In order to get compact notations, we denote by $(\Omega^\grast_\lie(\varP, \kg), \hd)$ the space of forms on this Lie algebroid $\Gamma_G(T\varP)$ with values in its kernel $\Gamma_G(\varP, \kg)$ and by $(\Omega^\grast_\lie(\varP), \hd_\lie)$ the space of forms with values in $C^\infty(\varM)$.

The local trivializations of the Lie algebroid $\Gamma_G(T\varP)$ can be explicitly obtained using local trivializations of the principal fiber bundle $\varP$. Let $\{(U_i, \varphi_i)\}$ be a system of trivializations of $\varP$ where $\varphi_i : \varP_{|U_i} \xrightarrow{\simeq} U_i \times G$ and denote by $s_i : U_i \rightarrow \varP$ with $s_i(x) = \varphi_i^{-1}(x,e)$ the associated local sections. One has $s_j(x) = s_i(x) g_{ij}(x)$ on any $U_{ij} = U_i \cap U_j \neq \ensvide$ where $g_{ij} : U_{ij} \rightarrow G$ are the associated transition functions. 

The isomorphism $\Psi_i : \Gamma(U_i \times \kg) \xrightarrow{\simeq} \Gamma_G(\varP,\kg)$ is such that $\Psi_i^{-1}(v) = s_i^\ast v$ for any $v \in \Gamma_G(\varP, \kg)$. Explicitly, one has $\Psi_i(\eta^i)(p) = \Ad_{g^{-1}} \eta^i(x)$ when $p=s_i(x)\cdotaction g$ for any $\eta^i \in \Gamma(U_i \times \kg)$.
With $p = s_i(x) \cdotaction g$, one defines $\nabla^{0, i}_{X\,|p} = T_{s_i(x)} \raR_{g} T_x s_i X_{| x} \in T_p \varP$. It is easy to show that $\raR_{g\,\ast} \nabla^{0, i}_{X} = \nabla^{0, i}_{X}$ for any $g \in G$ so that $\nabla^{0, i}_{X} \in \Gamma_G(TP)$. 
As a consequence, any $\sfX \in \Gamma_G(T\varP)$ is trivialized over $U_i$ as $X\oplus \gamma^i$ where $X \in \Gamma(U_i)$ is the restriction of a globally defined vector fields $X = \pi_\ast(\sfX)$ and $\gamma^i : U_i \rightarrow \kg$ represents the vertical part of $\sfX$ on $\varP$. More concretely, one has 
\begin{equation*}
\sfX = \nabla^{0, i}_{X} - \widehat{\gamma}^{i\,\varP} \in \Gamma_G(TP)
\end{equation*}
where for any $p = s_i(x) \cdotaction g$, one has $\widehat{\gamma}^i(p) = \Psi_i(\gamma^i)(p) = \Ad_{g^{-1}} \gamma^i(x)$ and $\widehat{\gamma}^{i\,\varP} = -\iota (\widehat{\gamma}^i)$. 

For any $p \in \varP_{|U_{ij}}$ with $U_{ij} = U_i \cap U_j \neq \ensvide$, a straightforward computation shows that
\begin{equation*}
\nabla^{0, i}_{X\,|p} = \nabla^{0, j}_{X\,|p} + \left[ g^{-1} g_{ij} \dd g_{ij}^{-1}(X) g \right]^\varP_{|p}
\end{equation*}
so that 
\begin{equation*}
\widehat{\gamma}^{i\,\varP}(p) = \widehat{\gamma}^{j\,\varP}(p) + \left[ g^{-1} g_{ij} \dd g_{ij}^{-1}(X) g \right]^\varP_{|p}.
\end{equation*}
This implies that
\begin{equation}
\label{eq-recollementgammai}
\gamma^i = g_{ij} \gamma^j g_{ij}^{-1} + g_{ij} \dd g_{ij}^{-1}(X),
\end{equation}
so that for Atiyah Lie algebroids one has
\begin{align}
\label{eq-alphachiatiyah}
\alpha_{j}^{i}(\gamma) &= g_{ij} \gamma g_{ij}^{-1}
&
\chi_{ij}(X) &= g_{ij} \dd g_{ij}^{-1}(X).
\end{align}

\medskip
We now suppose that $\kg$ is semisimple, so that its Killing form $k$ is non degenerate.

On any trivialization of $\Gamma_G(T\varP)$ associated to a trivialization $(U_i, \varphi_i)$ of $\varP$, we define $h^i_\loc(\gamma,\eta) = k(\gamma, \eta)$ for any $\gamma, \eta : U_i \rightarrow \kg$. Then, using the invariance of $k$ under the adjoint action of $G$ on $\kg$, we get that the $h^i_\loc$'s define a global metric $h$ on $\lieL$ (see \eqref{eq-relationtrivh} and \eqref{eq-alphachiatiyah}).

Let us introduce a fixed connection $\nabla$ on the Lie algebroid $\Gamma_G(T\varP)$, \textsl{i.e.} an ordinary connection on the principal fiber bundle $\varP$. The mixed basis on any trivialization of $\Gamma_G(T\varP)$ will be defined relative to this connection. 

We suppose that $\Gamma_G(T\varP)$ is inner orientable, \textsl{i.e.} that the vector bundle $\caL = \varP \times_\Ad \kg$ is orientable.

\begin{proposition}
\label{prop-atiyah-commutationofdifferentials}
In any local trivialization of $\Gamma_G(T\varP)$, the matrix with entries $h_{a b} = h_\loc(E_a, E_b)$ is constant so that $\sqrt{|h_\loc|}$ is locally constant.

Suppose that the Lie algebra $\kg$ is unimodular. Then for any $\omega \in \Omega^\grast_\lie(\varP)$ one has
\begin{equation*}
\int_\inner \hd_\lie \omega = \dd \int_\inner \omega
\end{equation*}
\end{proposition}

A Lie algebra is unimodular in our sense if the trace of its adjoint action vanishes. The usual definition of a unimodular group is that its left invariant Haar measure is equal to its right invariant Haar measure. When the group is finite dimensional and connected, this is equivalent to the fact that its adjoint action is of determinant $1$ or that its Lie algebra is unimodular in our sense \cite{Bour72b}.

There are well-known sufficient conditions for a group to be unimodular: compact, abelian, connected reductive or nilpotent…

This proposition is similar to Theorem~1.2 in \cite{MR1908998} or Theorem~5.2.2 in \cite{Kuba96a}. A key condition required in these Theorems is that the cross-section $\varepsilon \in \exter^n \lieL$ which defines the integral along the fiber be invariant oriented, which means that it is invariant under the $\exter^n \ad$ representation of $\lieA$ on $\exter^n \lieL$. Using the definition of $\varepsilon$ given by \eqref{eq-epsilonkubarski} in our context, this is equivalent to both $|h_\loc|$ being locally constant and the Lie algebra being unimodular. In the following proof, we will only use these two conditions. The first one is satisfied when $h$ is defined as the Killing form, but this is not a necessary condition.

\begin{proof}
The first assertion is a direct consequence of the fact that in any trivialization of $\Gamma_G(T\varP)$ the metric $h$ is defined using the (constant) Killing form $k$.

For the second assertion, one first shows that
\begin{equation}
\label{eq-soninnerlessthatn}
\ds (\theta^{a_1} \ordwedge \cdots \ordwedge \theta^{a_{n-1}}) = (-1)^{n} \tr (C_{a_n})\; \theta^{a_1} \ordwedge \cdots \ordwedge \theta^{a_n}
\end{equation}
where $\ds$ is the Chevalley-Eilenberg differential on $\exter^\grast \kg^\ast$, for which $\ds \theta^c = - \frac{1}{2} C^{c}_{a b} \theta^a \ordwedge \theta^b$, $C_{a_n}$ is the matrix $(C_{a_n a}^{b})_{a,b}$ and $a_n$ is the missing index in the fixed multi-index $(a_1, \dots, a_{n-1})$ with $a_k \neq a_\ell$ for $k \neq \ell$.

Using this result, let us now collect the factor of $\sqrt{|h_\loc|} \lfc^{1} \ordwedge \cdots \ordwedge \lfc^{n}$ in $(\dd + \ds) \omega_{\loc}$ when one uses the decomposition 
\begin{equation*}
\omega_{\loc} = (-1)^n \omega^\maxinner_{\loc} \sqrt{|h_\loc|} \lfc^{1} \ordwedge \cdots \ordwedge \lfc^{n} + \omega^R
\end{equation*}
with $\omega^\maxinner_{\loc} \in \Omega^\grast(U)$ and $\omega^R$ containing only terms of degrees $< n$ in the $\lfc^a$'s.

$\dd \sqrt{|h_\loc|} \lfc^{1} \ordwedge \cdots \ordwedge \lfc^{n}$ and $\dd \omega^R$ do not contribute because $\sqrt{|h_\loc|}$ is locally constant and the degrees do not match for other terms. When $\kg$ is unimodular, \eqref{eq-soninnerlessthatn} implies that $\ds \sqrt{|h_\loc|} \lfc^{1} \ordwedge \cdots \ordwedge \lfc^{n}$ and $\ds \omega^R$ do not contribute. The remaining term is then $\dd \omega^\maxinner_{\loc}$, which is globally $\dd \int_\inner \omega$.
\end{proof}

\begin{proposition}
Suppose that $h$ is constructed using the Killing form of a semi-simple unimodular Lie algebra $\kg$, suppose that $\varM$ is compact without boundary and suppose that $\Gamma_G(T\varP)$ is orientable. Then for any $\omega, \eta \in \Omega^\grast_\lie(\varP, \kg)$ such that $\hd \eta = 0$ one has $\langle \hd \omega, \eta \rangle = 0$.

There is a coupling in cohomology $\langle -, - \rangle : H^r \times H^{m+n-r} \rightarrow \gC$,
where $H^r = H^r(\Omega^\grast_\lie(\varP, \kg), \hd)$ is the cohomology of forms with values in the kernel.
\end{proposition}

\begin{proof}
The proof is a direct computation similar to the one in the proof of Proposition~\ref{prop-atiyah-commutationofdifferentials}.
\end{proof}

\smallskip
From now on,  we assume that the structure group $G$ is connected, simply connected, semi-simple, unimodular and of dimension $n$. In other words, $G$ is the connected and simply connected group associated to a semi-simple unimodular $n$-dimensionnel Lie algebra $\kg$.

\smallskip
The space $\kg_\equ = \{ \xi^\varP \oplus \xi \ / \ \xi \in \kg \}$ is a sub Lie algebra of $\tla(\varP, \kg)$, where $\xi^\varP \in \Gamma(T\varP)$ is the fundamental vector field associated to $\xi \in \kg$ for the right action of $G$ on $\varP$. $\kg_\equ$ defines a Cartan operation on the differential complex $(\Omega^\grast_\tla(\varP,\kg), \hd_\tla)$. Denote by $(\Omega^\grast_\tla(\varP,\kg)_{\kg_\equ}, \hd_\tla)$ the differential graded subcomplex of basic elements.

\begin{proposition}[\cite{Mass38}]
\label{prop-identificationdifferentialcalculusAtiyah}
$(\Omega^\grast_\lie(\varP, \kg), \hd)$ and $(\Omega^\grast_\tla(\varP,\kg)_{\kg_\equ}, \hd_\tla)$ are isomorphic as differential graded complexes.

The same is true for $\kg_\equ$-basic forms in $\Omega^\grast_\tla(\varP)$ and $\Omega^\grast(\lieA)$.
\end{proposition}

Let us associate to a global form $\omega \in \Omega^\grast_\lie(\varP, \kg)$ its family of local forms $\{\omega_{\loc}^i\}_{i \in I}$ with $\omega_{\loc}^i \in \Omega^\grast_\tla(U_i,\kg)$ satisfying \eqref{eq-relationtrivforms}. Let $\widehat{\omega} \in \Omega^\grast_\tla(\varP,\kg)_{\kg_\equ}$ be the $\kg_\equ$-basic form corresponding to $\omega \in \Omega^\grast_\lie(\varP, \kg)$ in the identification of Proposition~\ref{prop-identificationdifferentialcalculusAtiyah}.

\begin{lemma}
\label{lemma-pullbackbasicforms}
One has $\omega_{\loc}^i = s_i^\ast \widehat{\omega} \in \Omega^\grast_\tla(U_i,\kg)$.
\end{lemma}

\begin{proof}
Let us first recall some key features of the identification of the differential calculus $(\Omega^\grast_\lie(\varP, \kg), \hd)$ with $(\Omega^\grast_\tla(\varP,\kg)_{\kg_\equ}, \hd_\tla)$. In \cite{Mass38} a short exact sequence of Lie algebras and $C^\infty(\varM)$-modules
\begin{equation*}
\xymatrix@1{{\algzero} \ar[r] & {\caZ} \ar[r] & {\caN} \ar[r]^-{\rho_\varP} & {\Gamma_G(T\varP)} \ar[r] & {\algzero}}
\end{equation*}
is used where $\caZ$ is defined to be the $C^\infty(\varP)$-module generated by $\kg_\equ$ and $\caN = \Gamma_G(T\varP) \oplus \caZ \subset \tla(\varP, \kg)$. It is shown that $\caN$ generates the space $\tla(\varP, \kg)$ as a $C^\infty(\varP)$-module.
The isomorphism $\lambda : \Omega^\grast_\tla(\varP,\kg)_{\kg_\equ} \rightarrow \Omega^\grast_\lie(\varP, \kg)$ is explicitly defined as follows. For any $\widehat{\omega} \in \Omega^r_\tla(\varP,\kg)_{\kg_\equ}$, for any $\sfX_1, \dots, \sfX_r \in \Gamma_G(T\varP)$, denote by $\widehat{X}_1, \dots, \widehat{X}_r \in \caN$ any family such that $\rho_\varP(\widehat{X}_i) = \sfX_i$, then the map $p \mapsto \lambda(\widehat{\omega})(\sfX_1, \dots, \sfX_r)(p) = \widehat{\omega}(\widehat{X}_1, \dots, \widehat{X}_r)(p) \in \kg$ is a $G$-equivariant map.

In order to simplify the exposition, we prove the lemma for $1$-forms. The algebraic machinery is the same for $p$-forms. For any $\sfX \in \Gamma_G(T\varP)$ and any $x \in U_i$, by definition one has
\begin{equation*}
\omega_{\loc}^i (X \oplus \gamma^i)(x) 
= \Psi_{i}^{-1} (\omega(\sfX))(x) 
= \omega(\sfX)(s_i(x))
= \widehat{\omega}_{s_i(x)}( \widehat{X}_{|s_i(x)})
\end{equation*}
for any $\widehat{X} \in \caN$ such that $\rho_\varP(\widehat{X}) = \sfX$. On $\varP_{|U_i}$, let us take $\widehat{X} = \nabla^{0, i}_{X} - \widehat{\gamma}^{i\,\varP} + (\widehat{\gamma}^{i\,\varP} \oplus \widehat{\gamma}^{i}) = \nabla^{0, i}_{X} \oplus \widehat{\gamma}^{i}$ where $\widehat{\gamma}^{i\,\varP} \oplus \widehat{\gamma}^{i} \in \caZ$ \cite{Mass38}. Then $\widehat{X}_{|s_i(x)} = (s_{i\, \ast} X)_{|x} \oplus \gamma^i(x)$ by construction of $\nabla^{0, i}_{X}$ and $\widehat{\gamma}^i$. This gives
\begin{equation*}
\omega_{\loc}^i (X \oplus \gamma^i) = (s_i^\ast \widehat{\omega})(X \oplus \gamma^i)
\end{equation*}
\end{proof}

It is possible to show directly that the proposed family of local forms $\{ s_i^\ast \widehat{\omega} \}_{i \in I}$ defined in this lemma satisfies \eqref{eq-relationtrivforms} using the basicity of $\widehat{\omega}$. 

Let $\sfX \in \Gamma_G(T\varP)$ be trivialized over $U_i$ as $X \oplus \gamma^i$ for a global vector field $X$ on $\varM$. For simplicity, let us write $s=s_i$ and $g = g_{ij}$ so that $s_j = s g$. With these notations, on $U_{ij} = U_i \cap U_j \neq \ensvide$ one has $\gamma^i = \Ad_{g} \gamma^j + g \dd g^{-1}(X)$. Then one has
\begin{align*}
(s_j^\ast \widehat{\omega})_{| x}(X_{|x} \oplus \gamma^j(x)) 
&= \widehat{\omega}_{|s_j(x)} ( s_{j\,\ast} X_{|x} \oplus \gamma^j(x))
\\
&= \widehat{\omega}_{|s(x) g(x)} ( \raR_{g\,\ast} s_{\ast} X_{|x} + \left[ g^{-1} \dd g(X) \right]^\varP_{|s(x) g(x)} \oplus \gamma^j(x))
\end{align*}
using standard results in differential geometry \cite{KobaNomi96a}.

The $\kg_\equ$-horizontality property of $\widehat{\omega}$ implies that at the level of the Lie group $G$ (connected and simply connected) one has
\begin{equation*}
\widehat{\omega}_{|s(x) g(x)} \left( \left[ g^{-1} \dd g(X) \right]^\varP_{|s(x) g(x)} \oplus \left[ g^{-1} \dd g(X) \right]_{|x} \right) = 0
\end{equation*}
so that
\begin{align*}
\omega_{\loc | x}^j(X_{|x} \oplus \gamma^j(x)) 
&= \widehat{\omega}_{|s(x) g(x)} ( \raR_{g\,\ast} s_{\ast} X_{|x} \oplus \gamma^j(x) - \left[ g^{-1} \dd g(X) \right]_{|x})
\\
&= \widehat{\omega}_{|s(x) g(x)} ( \raR_{g\,\ast} s_{\ast} X_{|x} \oplus \Ad_{g^{-1}}\gamma^i(x))
\end{align*}
The $\kg_\equ$-invariance property of $\widehat{\omega}$ implies that
\begin{equation*}
\widehat{\omega}_{|s(x) g(x)} ( \raR_{g\,\ast} s_{\ast} X_{|x} \oplus \Ad_{g^{-1}}\gamma^i(x))
=
\Ad_{g^{-1}}\, \widehat{\omega}_{|s(x)} ( s_{\ast} X_{|x} \oplus \gamma^i(x))
\end{equation*}
so that
\begin{equation*}
\omega_{\loc | x}^j(X_{|x} \oplus \gamma^j(x))  = \Ad_{g^{-1}}\, \omega_{\loc | x}^i ( X_{|x} \oplus \gamma^i(x))
\end{equation*}
which is \eqref{eq-relationtrivforms}.

The identifications between these differential calculi can be summarized in the following diagram:
\begin{equation*}
\xymatrix@R=10pt@C=10pt{
 & & & {\text{\parbox{5em}{\centering Trivial Lie\\ Algebroids}}} & \\
{\Omega^\grast_\lie(\varP, \kg)} \ar@{^{(}->}[rrr]_-{\text{inclusion}} \ar[dddrrr]_(0.45)*!/^-3pt/{\text{\scriptsize trivialization}} 
& & & 
{\Omega^\grast_\tla(\varP,\kg)_{\kg_\equ}} \ar[ddd]^-{\{ s_i^\ast\}} \ar@/_1.2pc/[lll]_-{\lambda}
& {\text{Global description}} \\
& & & & \\
& & & & \\
& & & {\prod_{i\in I} \Omega^\grast_\tla(U_i,\kg)} & {\text{Local description}}
}
\end{equation*}

\medskip

\begin{corollary}
\label{cor-pullbackvolumeform}
The $\kg_\equ$-basic form in $\Omega^\grast_\tla(\varP)$ corresponding to $\omega_{h,\lfc} \in \Omega^\grast(\lieA)$ is
\begin{equation*}
\widehat{\omega}_{k,\nabla} = (-1)^n \sqrt{|k|} (\omega_\nabla^{1} - \theta^1) \ordwedge \cdots \ordwedge (\omega_\nabla^{n} - \theta^n)
\end{equation*}
where $\omega_\nabla = \omega_\nabla^{a} \otimes E_a \in \Omega^\grast(\varP) \otimes \kg$ is the (ordinary) connection $1$-form on $\varP$ associated to $\nabla$.
\end{corollary}

\begin{proof}
Notice that the proof of Lemma~\ref{lemma-pullbackbasicforms} applies to basic forms in $\Omega^\grast_\tla(\varP)$ which are trivialized as local forms in $\Omega^\grast_\tla(U_i) = \Omega^\grast(U_i) \otimes \exter^\grast \kg^\ast$ via the pull-back $s_i^\ast$.

Using the unimodular property of $\kg$, a straightforward computation shows that the proposed expression for $\widehat{\omega}_{k,\nabla}$ is $\kg_\equ$-basic.

$\omega_{h,\lfc}$ is locally defined on $U_i$ as $(-1)^n \sqrt{|h^i_\loc|} \lfc_i^{1} \ordwedge \cdots \ordwedge \lfc_i^{n}$ where $h^i_\loc = k$ and $\lfc_i^a = A_i^a - \theta^a \in \Omega^1_\tla(U_i)$. $A_i = A_i^a \otimes E_a$ is the local expression of $\omega_\nabla$ given explicitly by $A_i = s_i^\ast \omega_\nabla$.
\end{proof}

\medskip
For any $\omega = \omega_{\dR} \otimes \omega_{\algebraic} \otimes \xi \in \Omega^\grast(\varP) \otimes \exter^\grast \kg^\ast \otimes \kg = \Omega^\grast_\tla(\varP,\kg)$, we define a natural map $\int_{\algebraic} \Omega^\grast_\tla(\varP,\kg) \rightarrow \Omega^{\grast-n}(\varP) \otimes \kg$ by
\begin{equation*}
\int_{\algebraic} \omega  = 
\begin{cases}
\omega_{\dR} \otimes \xi & \text{ if $\omega = \omega_{\dR} \otimes \sqrt{|k|}\, \theta^1 \ordwedge \cdots \ordwedge \theta^n \otimes \xi$}\\
0 & \text{ if $\omega_{\algebraic} \not\in \exter^n \kg^\ast$}
\end{cases}
\end{equation*}

Here we identify the space of forms $\Omega^{\grast}(\varM, \varL)$ with the space of tensorial forms in $\Omega^{\grast}(\varP) \otimes \kg$ \cite{KobaNomi96a}.

\begin{proposition}
\label{prop-relationsdeRhamTLAAtiyah}
The following diagram is commutative

\begin{equation*}
\xymatrix@R=8ex{
  {\Omega^\grast_\lie(\varP, \kg)} \ar@{^{(}->}[r] \ar[d]_-{\int_\inner}
& {\Omega^\grast_\tla(\varP,\kg)} \ar[d]^-{\int_{\algebraic}} 
\\
  {\Omega^{\grast-n}(\varM, \varL)} \ar@{^{(}->}[r] 
& {\Omega^{\grast-n}(\varP) \otimes \kg} 
 }
\end{equation*}
\end{proposition}

\begin{proof}
The first point to check is that $\int_{\algebraic}$ maps $\kg_\equ$-basic forms in $\Omega^\grast_\tla(\varP,\kg)$ to tensorial forms in $\Omega^{\grast-n}(\varP) \otimes \kg$. Because $G$ is connected and simply connected, a form $\sum_a \widehat{\omega}^a_{\dR} \otimes \xi_a \in \Omega^{\grast}(\varP) \otimes \kg$ is tensorial if and only if $ \sum_a (L_{\xi^\varP} \widehat{\omega}^a_{\dR}) \otimes \xi_a + \sum_a \widehat{\omega}^a_{\dR} \otimes [\xi, \xi_a] = 0$ and $ \sum_a (i_{\xi^\varP}\widehat{\omega}^a_{\dR}) \otimes \xi_a = 0$ for any $\xi \in \kg$.

A form $\widehat{\omega} = \sum_a \widehat{\omega}^a_{\dR} \otimes \sqrt{|k|}\, \theta^1 \ordwedge \cdots \ordwedge \theta^n \otimes \xi_a$ is basic if and only if for any $\xi \in \kg$ one has
\begin{multline*}
\sum_a (L_{\xi^\varP} \widehat{\omega}^a_{\dR}) \otimes \sqrt{|k|}\, \theta^1 \ordwedge \cdots \ordwedge \theta^n \otimes \xi_a 
\\
+ 
\sum_a \widehat{\omega}^a_{\dR} \otimes (L^\kg_\xi \sqrt{|k|}\, \theta^1 \ordwedge \cdots \ordwedge \theta^n )\otimes \xi_a
\\
+
\sum_a \widehat{\omega}^a_{\dR} \otimes \sqrt{|k|}\, \theta^1 \ordwedge \cdots \ordwedge \theta^n \otimes [\xi,\xi_a] = 0
\end{multline*}
and
\begin{multline*}
\sum_a (i_{\xi^\varP}\widehat{\omega}^a_{\dR}) \otimes \sqrt{|k|}\, \theta^1 \ordwedge \cdots \ordwedge \theta^n \otimes \xi_a
\\
+
\sum_a (-1)^{|\widehat{\omega}^a_{\dR}|} \widehat{\omega}^a_{\dR} \otimes (i_\xi \sqrt{|k|}\, \theta^1 \ordwedge \cdots \ordwedge \theta^n ) \otimes \xi_a
=0.
\end{multline*}
Because $\kg$ is unimodular, one has $L^\kg_\xi \sqrt{|k|}\, \theta^1 \ordwedge \cdots \ordwedge \theta^n = 0$, so that $\int_{\algebraic} \widehat{\omega} = \sum_a \widehat{\omega}^a_{\dR} \otimes \xi_a$ is invariant. Looking at each bidegrees for the horizontally condition on $\widehat{\omega}$, one gets $\sum_a (i_{\xi^\varP}\widehat{\omega}^a_{\dR}) \otimes \sqrt{|k|}\, \theta^1 \ordwedge \cdots \ordwedge \theta^n \otimes \xi_a = 0$ so that $\int_{\algebraic} \widehat{\omega}$ is horizontal.

The second point to check is that $\int_{\algebraic}$ coincides on $\kg_\equ$-basic forms with $\int_\inner$. In order to do that, we consider these integrations on a trivialization of $\varP$ given by a local section $s : U \rightarrow \varP$. Then one has the following diagram:
\begin{equation*}
\xymatrix@R=8ex@C=30pt{
  {\Omega^\grast_\tla(\varP,\kg)_{\kg_\equ}} \ar[d]_-{\int_{\algebraic}} \ar[r]^-{s^\ast}
& {\Omega^\grast_\tla(U,\kg)} \ar[d]^-{\int_\inner}
\\
  {(\Omega^{\grast-n}(\varP)\otimes \kg)_{\text{tensorial}}} \ar[r]^-{s^\ast}  
& {\Omega^{\grast-n}(U) \otimes \kg} 
 }
\end{equation*}
The map $s^\ast$ (see for instance \cite{KobaNomi96a}) in the bottom row is the same as the map $s^\ast$ in the top row. For any basic form $\widehat{\omega} = \sum_a (-1)^n \sqrt{|k|}\, \widehat{\omega}^a_{\dR} (\omega_\nabla^{1} - \theta^1) \ordwedge \cdots \ordwedge (\omega_\nabla^{n} - \theta^n) \otimes \xi^a \in \Omega^\grast_\tla(\varP,\kg)_{\kg_\equ}$ one has 
$s^\ast \widehat{\omega} = \sum_a (-1)^n \sqrt{|k|}\, (s^\ast\widehat{\omega}^a_{\dR})\, \lfc^{1} \ordwedge \cdots \ordwedge \lfc^{n}  \otimes \xi^a$ 
because of Corollary~\ref{cor-pullbackvolumeform}, so that
$\int_\inner s^\ast \widehat{\omega} = \sum_a (s^\ast\widehat{\omega}^a_{\dR}) \otimes \xi^a$. 
On the other hand, one has
$\int_{\algebraic} \widehat{\omega} = \sum_a \widehat{\omega}^a_{\dR} \otimes \xi^a$
so that
$s^\ast \int_{\algebraic} \widehat{\omega} = \sum_a (s^\ast\widehat{\omega}^a_{\dR}) \otimes \xi^a$.

This proves the coincidence of the two integrals, because $\int_\inner$ on $\Omega^\grast_\lie(\varP, \kg)$ is completely determined by $\int_\inner$ on the trivializations of forms in $\Omega^\grast_\tla(U,\kg)$.
\end{proof}

Using Proposition~\ref{prop-tripleforinnernondegeneratemetric}, one can define an inner non degenerate metric $\hg$ on $\Gamma_G(T\varP)$ as a triple $(g, h, \nabla)$ where $h$ and $\nabla$ are defined as above and $g$ is an ordinary metric on the base manifold $\varM$. Then the properties of this triple are exactly the ones defining a metric for a non-abelian Kaluza-Klein theory on $\varP$ \cite{Kerner1988fn}. Notice that the geometrical point of view is generally adopted in these theories (geodesics and trajectories of particle) while our point of view here is the one from field theories.

\smallskip
When the Lie algebra $\kg$ is reductive, the cohomology space $H^\grast(\kg;\kg)$, which appears in the spectral sequence of Section~\ref{sec-cohomology}, identifies with $\kg_\inv \otimes (\exter^\grast \kg^\ast)_\inv$ where $\inv$ specifies invariant elements for the $\ad$-representation on $\kg$ and its adjoint representation $\ad^\ast$ on $\exter^\grast \kg^\ast$ \cite{GreuHalpVans76}.

For Atiyah Lie algebroids, the automorphisms $\halpha_U^V$ are induced by the adjoint action of $G$ on $\kg$ and its coadjoint on $\kg^\ast$. When $G$ is connected and simply connected, this implies that the presheaf $\caH^\grast = H^\grast(\kg;\kg)$ is a constant presheaf, so that there exists a spectral sequence $\{ ({E}_i, \dd_i) \}$ such that
\begin{align*}
E_2^{p,q} &\simeq H^p(\varM) \otimes H^q(\kg;\kg)
\\
E_\infty &\simeq H(\Omega^\grast_\lie(\varP, \kg))
\end{align*}

Another spectral sequence involving the cohomology $(\Omega^\grast_\lie(\varP, \kg), \hd)$ is defined by using the identification of $\Omega^\grast_\lie(\varP, \kg)$ as the basic subcomplex of $(\Omega^\grast_\tla(\varP,\kg), \hd_\tla)$. This spectral sequence $\{ (\widetilde{E}_i, \dd_i) \}$ is associated to the filtration:
\begin{equation*}
F^p(\Omega^q_\tla(\varP,\kg)) = \{\ \widehat{\omega} \in \Omega^q_\tla(\varP,\kg) \ / \ i_{\widehat{X}_1} \cdots i_{\widehat{X}_{q-p+1}} \widehat{\omega} = 0 \ \text{for any } \widehat{X}_k \in \kg_\equ \ \} 
\end{equation*}
Using general results about such a filtration of a Cartan operation \cite{GreuHalpVans76}, one has
\begin{align*}
\widetilde{E}^{p,0}_0 & \simeq \Omega^p_\tla(\varP,\kg)_{\kg_\equ-\hor}
\\
\widetilde{E}^{p,0}_1 &\simeq \Omega^p_\tla(\varP,\kg)_{\kg_\equ}
\\
\widetilde{E}^{p,0}_2 &\simeq H^p(\Omega^\grast_\tla(\varP,\kg)_{\kg_\equ}) = H^p(\Omega^\grast_\lie(\varP, \kg))
\\
\widetilde{E}_\infty &\simeq H^\grast(\Omega^\grast_\tla(\varP,\kg))
\end{align*}
with $H^\grast(\Omega^\grast_\tla(\varP,\kg)) = H^\grast(\varP) \otimes H^\grast(\kg;\kg)$ and ``$\kg_\equ-\hor$'' designates the horizontal elements for the $\kg_\equ$-operation.

\section{Derivations on a vector bundle}
\label{sec-Derivationsonavectorbundle}

Let $\varE$ be a rank $p$ complex vector bundle over the manifold $\varM$. Using any hermitian structure on $\varE$, we suppose that its structure group $H$ is contained in $U(p)$, the group of complexe unitary $p\times p$ matrices. Denote by $\End(\varE) = \varE \otimes \varE^\ast$ the fiber bundle of endomorphisms of $\varE$ where $\varE^\ast$ is the dual vector bundle associated to $\varE$. Denote by $\algA(\varE) = \Gamma(\End(\varE))$ the algebra of endomorphisms of $\varE$.

Let $\kD(\varE)$ be the space of first order operators on $\Gamma(\varE)$ whose symbol is the identity. Then
\begin{equation*}
\xymatrix@1{{\algzero} \ar[r] & {\algA(\varE)} \ar[r]^-{\iota} & {\kD(\varE)} \ar[r]^-{\sigma} & {\Gamma(T \varM)} \ar[r] & {\algzero}}
\end{equation*}
is the transitive Lie algebroid of derivations of $\varE$ where $\sigma$ is the symbol map \cite{MR585879}.

Let us denote by $(\Omega^\grast_\lie(\varE, \algA(\varE)), \hd)$ the graded differential algebra of generalized forms on this transitive Lie algebroid with values in the kernel, and let us denote by $(\Omega^\grast_\lie(\varE), \hd_\lie)$ the graded commutative differential algebra of forms on $\kD(\varE)$ with values in $C^\infty(\varM)$. The natural inclusion $C^\infty(\varM) \rightarrow \algA(\varE)$ induces a morphism of graded differential algebras
\begin{equation*}
\Omega^\grast_\lie(\varE) \hookrightarrow \Omega^\grast_\lie(\varE, \algA(\varE))
\end{equation*}

Let us just collect the formulas for the local description of $\kD(\varE)$. Let $\{ U_i, \phi_i \}_{i \in I}$ be a system of trivializations of $\varE$ associated to a good cover $\{ U_i \}_{i \in I}$ of $\varM$, where $\phi_i : U_i \times \gC^p \rightarrow \varE_{|U_i}$ are linear isomorphisms. Then this system of trivializations induces a natural system of trivializations of $\End(\varE)$, $\{ U_i, \widehat{\phi}_i \}_{i \in I}$, such that $\widehat{\phi}_i : U_i \times M_p(\gC) \rightarrow \End(\varE)_{|U_i}$ and $\widehat{\phi}_i(x,\gamma) \cdotaction \phi_i(x,v) = \phi_i(x, \gamma \cdotaction v)$ for any $\gamma \in M_p(\gC)$ and any $v \in \gC^p$. Any $s \in \Gamma(\varE)$ (resp. $a \in \algA(\varE)$) is then trivialized by a family of maps $s^i : U_i \rightarrow \gC^p$ (resp. $a^i : U_i \rightarrow M_p(\gC)$). A first order operator $\kX \in \kD(\varE)$ is trivialized as a family of elements $X_i \oplus \gamma^i \in \Gamma(T U_i) \oplus \Gamma(U_i \times M_p(\gC))$ through the relation
\begin{equation*}
(\kX \cdotaction s)(x) = \phi_i \left(x, (X_i \cdotaction s_i)(x) + \gamma^i(x) \cdotaction s_i(x) \right)
\end{equation*}
for any $x \in U_i$ where $\cdotaction$ means either the action of a vector field on vector valued functions or the action of matrices on vectors. Notice that $X_i = X_j = X$ on $U_{ij} = U_i \cap U_j \neq \ensvide$ where $X = \sigma(\kX)$ and $\gamma^i = h_{ij} \gamma^j h_{ij}^{-1} + h_{ij} \dd h_{ij}^{-1} (X)$ where $h_{ij} : U_{ij} \rightarrow H \subset U(p)$ are the transition functions of $\varE$ such that $s^i(x) = h_{ij}(x) s^j(x)$ for any $x \in U_{ij}$.
The system of trivializations considered for $\kD(\varE)$ is thus defined by
\begin{align*}
\nabla^{0,i}_X &= X,
&
\Psi_i(a^i) &= a^i \text{ for any $a^i : U_i \rightarrow M_p(\gC)$},
\end{align*}
and one has
\begin{align*}
\alpha_j^i(\gamma) &= h_{ij} \gamma h_{ij}^{-1},
&
\chi_{ij}(X) &= h_{ij} \dd h_{ij}^{-1}(X).
\end{align*}

The Lie algebra on which this Lie algebroid is modelled is $\kg = M_p(\gC) = M_p$ with the commutator as Lie bracket and $n= p^2$. This Lie algebra decomposes as $\kg = \gC \bbbone_p \oplus \ksl_n$ where $\bbbone_p$ is the unit matrix in $M_p$ and $\ksl_p$ is the Lie algebra of traceless matrices in $M_p$.

The local description of $\Omega^\grast_\lie(\varE, \algA(\varE))$ is given by the differential calculi $\Omega^\grast(U_i) \otimes \exter^\grast M_p^\ast \otimes M_p$ and the maps $\halpha_{i}^{\,j} : \Omega^\grast(U_i) \otimes \exter^\grast M_p^\ast \otimes M_p \rightarrow \Omega^\grast(U_j) \otimes \exter^\grast M_p^\ast \otimes M_p$ (on $U_i \cap U_j \neq \ensvide$) are morphisms of graded differential algebras.

In any trivialization, one can decompose any element $X \oplus \gamma^i = \tla(U_i, M_p)$ as $X \oplus \left( \frac{1}{p} \lambda^i \bbbone_p \oplus \gamma_0^i \right)$ where $\lambda^i = \tr(\gamma^i)$ and $\gamma_0^i = \gamma^i - \frac{1}{p} \lambda^i \bbbone_p : U_i \rightarrow \ksl_p$. Then the family $X \oplus \lambda^i$ associated to a family $X_i \oplus \gamma^i$ of trivializations of an element $\kX \in \kD(\varE)$ defines a global element in the transitive Lie algebroid
\begin{equation*}
\xymatrix@1{{\algzero} \ar[r] & {C^\infty(\varM)} \ar[r]^-{\iota} & {\kD(\det(\varE))} \ar[r]^-{\sigma} & {\Gamma(T \varM)} \ar[r] & {\algzero}}
\end{equation*}
where $\det(\varE) = \exter^p \varE$ is the determinant line bundle associated to $\varE$. This map is the natural representation of $\kD(\varE)$ on $\det(\varE)$ given by $\kX \mapsto \exter^p \kX$ where
\begin{equation*}
\left( \exter^p \kX \right)(e_1 \ordwedge \cdots \ordwedge e_p) = \sum_{k=1}^p e_1 \ordwedge \cdots \ordwedge \kX(e_k) \ordwedge \cdots \ordwedge e_p.
\end{equation*}
The induced map $\algA(\varE) \rightarrow C^\infty(\varM)$ is the globally defined trace $\tr$, which is a Lie morphism. This representation gives rise to a natural morphism of graded commutative differential algebras
\begin{equation*}
\Omega^\grast_\lie(\det(\varE)) \rightarrow \Omega^\grast_\lie(\varE).
\end{equation*}

The inner orientability of $\kD(\varE)$ corresponds to the orientability of the (vector) bundle $\End(\varE)$. Because $U(p)$ is unimodular, $\End(\varE)$ is always orientable. The trace map given before defines a non degenerate inner metric $h : \algA(\varE) \otimes_{C^\infty(\varM)} \algA(\varE) \rightarrow C^\infty(\varM)$ given by $h(a,b) = \tr(ab)$. In any local trivialization, the inner metric $h$ is represented by a constant matrix.

Notice that the unimodularity of the (real) Lie algebra $\ku_p$ of $U(p)$ implies the unimodularity of the (complex) Lie algebra $\kg = M_p$.

The inner integration $\int_\inner : \Omega^\grast_\lie(\varE, \algA(\varE)) \rightarrow \Omega^{\grast-n}(\varM, \End(\varE))$ defined by the inner metric $h$ can be composed with the trace map in order to define
\begin{equation}
\label{eq-definnerintegrationtrace}
\int^{\tr}_\inner = \tr \circ \int_\inner : \Omega^\grast_\lie(\varE, \algA(\varE)) \rightarrow \Omega^{\grast-n}(\varM)
\end{equation}

\begin{proposition}
For any $\omega \in \Omega^\grast_\lie(\varE, \algA(\varE))$ one has
\begin{equation*}
\int^{\tr}_\inner \hd \omega = \dd \int^{\tr}_\inner \omega
\end{equation*}
\end{proposition}

\begin{proof}
The differential $\hd$ is locally the sum of three parts on $\Omega^\grast(U_i) \otimes \exter^\grast M_p^\ast \otimes M_p$: the de~Rham differential on $\Omega^\grast(U_i)$, the Chevalley-Eilenberg differential on $\exter^\grast M_p^\ast$ and the adjoint action on $M_p$. Using similar arguments as the ones used in the proof of Proposition~\ref{prop-atiyah-commutationofdifferentials} and the fact that the trace kills the adjoint action on $M_p$, one gets the result.
\end{proof}

In \cite{Mass14}, a natural surjection $\kD(\varE) \rightarrow \der(\algA(\varE))$ was proposed: it associates to any $\kX \in \kD(\varE)$ the derivation $a \mapsto [\kX, a]$ for any $a \in \algA(\varE)$ where the commutator takes place in the space of operators on $\Gamma(\varE)$. Locally this corresponds to $X \oplus \gamma \mapsto X \oplus \ad_\gamma$.

If the structure group of $\varE$ can be reduced such that $H \subset SU(p)$, then there is a natural injection $\der(\algA(\varE)) \rightarrow \kD(\varE)$ of Lie algebroids defined locally by $X \oplus \ad_{\gamma^i} \mapsto X \oplus \gamma^i$ for any traceless $\gamma^i : U_i \rightarrow \ksl_p$ where $\der(\algA(\varE))$ is the Lie algebra and $C^\infty(\varM)$-module of derivations of the associative algebra $\algA(\varE)$. One then has a splitting of Lie algebroids
\begin{equation*}
\kD(\varE) \simeq \der(\algA(\varE)) \oplus C^\infty(\varM).
\end{equation*}
The inclusion of $\der(\algA(\varE))$ into $\kD(\varE)$ induces a natural morphism of graded differential algebras $\Omega^\grast_\lie(\varE,\algA(\varE)) \rightarrow \Omega^\grast_\der(\algA(\varE))$ where $\Omega^\grast_\der(\algA(\varE))$ is the derivation based differential calculus associated to $\algA(\varE)$. See \cite{Mass30,Mass38} for details concerning the precise definitions and properties of these noncommutative structures. This morphism connects together the integration $\int^{\tr}_\inner$ defined in \eqref{eq-definnerintegrationtrace} and a similar integration defined in the context of the noncommutative geometry of the algebra $\algA(\varE)$ (see \cite{Mass15}).

\bibliography{Forms-Lie-algebroids}

\end{document}